\documentclass[13pt]{amsart}
\usepackage{amsmath}
\usepackage{}
\usepackage{amsfonts}
\usepackage{amsfonts,latexsym,rawfonts,amsmath,amssymb,amsthm}
\usepackage[plainpages=false]{hyperref}

\usepackage{pdfsync}
\usepackage{graphicx}

\numberwithin{equation}{section}
\newcommand{\beq}{\begin{equation}}
\newcommand{\eeq}{\end{equation}}
\newcommand{\beqs}{\begin{eqnarray*}}
\newcommand{\eeqs}{\end{eqnarray*}}
\newcommand{\beqn}{\begin{eqnarray}}
\newcommand{\eeqn}{\end{eqnarray}}
\newcommand{\beqa}{\begin{array}}
\newcommand{\eeqa}{\end{array}}

\def\lra{\longrightarrow}

\def\bc{\begin{center}}
\def\ec{\end{center}}

\def\begeq{\begin{equation}}
\def\endeq{\end{equation}}
\def\and{\quad{\rm and}\quad}

\let\lra=\longrightarrow

\def\mapright\#1{\,\smash{\mathop{\lra}\limits^{\#1}}\,}

\newtheorem{prop}{Proposition}[section]
\newtheorem{theo}[prop]{Theorem}
\newtheorem{lem}[prop]{Lemma}

\newtheorem{cor}[prop]{Corollary}
\newtheorem{rem}[prop]{Remark}

\newtheorem{defi}[prop]{Definition}

\begin{document}

\title{${\rm K}$-energy on polarized compactifications of Lie groups}

%\date{-}

\author{Yan Li\ \ \ \ \ \ Bin  $\text{Zhou}^*$\ \ \ \ \ \ Xiaohua  $\text{Zhu}^{**}$}

\address{School of Mathematical Sciences, Peking
University, Beijing 100871, China.}

\email{liyanmath@pku.edu.cn\ \ \ bzhou@pku.edu.cn\ \ \ xhzhu@math.pku.edu.cn}

\thanks {*Partially supported by NSFC 11571018 and 11331001}

\thanks {** Partially supported by NSFC Grants 11271022 and 11331001.}

\begin{abstract}
In this paper, we study  Mabuchi's K-energy on a  compactification $M$  of a reductive Lie group $G$,  which is a complexification of its maximal compact subgroup $K$.  We give a criterion for  the properness of K-energy
on the space of ${ K\times K}$-invariant K\"ahler potentials. In particular,  it turns to give an alternative proof of Delcroix's theorem  for  the existence  of K\"ahler-Einstein metrics  in case of Fano manifolds  $M$.  We also study the existence of  minimizers of  K-energy for general  K\"ahler classes of $M$.
\end{abstract}

\maketitle

\section{Introduction}

The famous Yau-Tian-Donaldson's conjecture for the existence  of K\"ahler-Einstein metrics on Fano manifolds asserts that the  existence is equivalent to the K-stability. The conjecture has been recently solved by Tian \cite{T4}. Chen, Donaldson and Sun also give an alternative proof \cite{CDS}.   The notion of  K-stability was first introduced by Tian  by using special degenerations  \cite{Ti}  and  then reformulated by Donaldson in algebraic geometry  via  test-configurations \cite{D}. For both special degenerations  and  test-configurations, one has to study an infinite number of possible degenerations of the manifold.  A natural  question is how to verify the  K-stability by reducing it to a finite dimensional progress. The answer is known for Fano surfaces by Tian \cite{T1} and for toric Fano manifolds by Wang and Zhu \cite{WZ} (see also \cite{WZZ}). In fact, in both cases the existence is equivalent to the vanishing of Futaki invariant.

More recently,   Delcroix extends Wang-Zhu's result to a polarized compactification $M$ of a reductive Lie group $G$ with $c_1(M)>0$ \cite{Del2}. We call $M$ a {\it (bi-equivariant) compactification of $G$} if it admits a holomorphic $G\times G$ action   on $M$ with an open and dense orbit isomorphic to $G$ as a $G\times G$-homogeneous space.  $(M, L)$ is called a {\it polarized compactification} of $G$   if  $L$ is a $G\times G$-linearized ample line bundle on $M$. For more examples besides the toric manifolds, see \cite{AK, Del2, Del3}.

Let $T^\mathbb C$  be a  $r$-dimensional  maximal complex torus of $G$ with dimension $n$  and $\mathfrak M$  its group of characters.
Assume that $\Phi$ is the root system of $(G,T^\mathbb C)$ in $\mathfrak M$ and $\Phi_+$ is a chosen set of positive roots.  Let  $P$  be the polytope  associated to $(M,L)$, and $P_+$   the part of $P$ defined by $\Phi_+$. Denote by $2P_+$ its dilation at rate $2$.
Let $\rho={\frac 1 2}\sum_{\alpha\in\Phi_+}\alpha$ and $\Xi$ be the relative interior of the cone generated by $\Phi_+$.   Then Delcroix proved

\begin{theo}\label{de}  Let $M$ be a polarized compactification of $G$ with $c_1(M)>0$.
Then  $M$ admits a  K\"ahler-Einstein metric if and only if
\begin{equation}\label{bar}
bar\in 4\rho+\Xi,
\end{equation}
where
$bar=\frac{\int_{2P_+}y\pi(y) \,dy}{\int_{2P_+}\pi(y) \,dy}$
is the barycentre of $2P^+$ with respect to the weighted measure $\pi(y)dy$ and  $\pi(y)=\prod_{\alpha\in\Phi_+}\langle\alpha,y\rangle^2$.
\end{theo}

It is pointed by Delcroix that  (\ref{bar}) implies that the Futaki invariant vanishes for holomorphic vector fields induced by $G\times G$, but the inverse is not true in general. Thus one may ask if (\ref{bar}) is related to the K-stability and  is determined by a generalized Futaki invariant for some test-configurations. In the present paper, we will answer this question. In fact, motivated by the study on toric manifolds \cite{D}, we investigate the K-energy on the space of  ${ K\times K}$-invariant K\"ahler potentials  through the reduced K-energy $\mathcal K(\cdot)$ via Legendre transformation. We show that condition  (\ref{bar}) comes  from our formula of  $\mathcal K(\cdot)$ naturally when $c_1(M)>0$ (cf. Proposition \ref{5104}, Proposition \ref{necessary}). Moreover, we give an  alternative  proof of Theorem \ref{de} by showing the properness of the K-energy (cf. Section 4). The  K\"ahler-Ricci solitons case can be discussed similarly (cf. Section 5).

The main purpose of this paper is to give a criterion for  the properness of the K-energy on a  general  polarized compactification  $(M,L)$ of $G$ as done on  a toric manifold in  \cite{ZZ}. We divide $\partial{(2P_+)}\cap\partial (2P)$ into several pieces
$\{F_A\}_{A=1}^{d_0}$ such that  for any $A$, $F_A$ lies on an $(r-1)$-dimensional  hyperplane defined by $\langle y,u_A\rangle=\lambda_A$ for some primitive $u_A\in\mathfrak N$, where $\mathfrak N$ is the $\mathbb Z$-dual of $\mathfrak M$.
Define a cone by  $E_A=\{ty|~t\in[0,1],\,y\in F_A\}$ for any $A$.
It is clear that  $2P_+=\displaystyle\bigcup_{A=1}^{d_0} E_A$. Let
\begin{equation}\label{5106}
\Lambda_A={\frac{2}{\lambda_A}}\left(1+\langle2\rho,u_A\rangle\right).
\end{equation}
 Then the average of  scalar curvature $\bar S$ of $\omega_0\in2\pi c_1(L)$ is given by\footnote{(\ref{5304})  will be  verified  at the end of Section 2.}
\begin{equation}\label{5304}
\bar S={\frac{n\sum_{A}\Lambda_{A}\int_{ E_{A}}\pi \,dy}{\int_{2P_+}\pi \,dy}}.
\end{equation}

Define a weighted barycentre $\widetilde{bar}$ of $2P_+$ by
\begin{equation}\label{+5100}
\widetilde{bar}={\frac{\sum_A\Lambda_A\int_{ E_A}y\pi \,dy}{\sum_A\Lambda_A\int_{ E_A}\pi \,dy}}.
\end{equation}
Note that both  $bar$  and $\widetilde{bar}$ are in the dual space $\mathfrak a^*$ of $\mathfrak a$, where $\mathfrak a$ is the non-compact part of  Lie algebra $\mathfrak t^{\mathbb C}$ of  $T^\mathbb C$.
Denote by $bar_{ss}$ and $\widetilde{bar}_{ss}$ the projections of $bar$ and $\widetilde{bar}$ on the semisimple part $\mathfrak a_{ss}^*$ of  $\mathfrak a^*$, respectively. We prove

\begin{theo}\label{LZZ1}
Let $(M,L)$ be a polarized compactification of $G$ with vanishing Futaki invariant, and $\omega_0\in 2\pi c_1(L)$  a $K\times K$-inariant K\"ahler metric. Suppose that the polytope $2P_+$ satisfies the following conditions,
\begin{eqnarray}
&&\left(\min_A\Lambda_A\cdot\widetilde{bar}_{ss}-4\rho\right)\in\Xi, \label{tildebar1}\\[4pt]
&&\left(\widetilde{bar}_{ss}-bar_{ss}\right)\in\bar{\Xi}, \label{tildebar2}\\
&&(n+1)\cdot\min_A\Lambda_A-\bar S>0. \label{barS}
\end{eqnarray}
Then the K-energy $\mu_{\omega_0}(\cdot)$  is proper  on $\mathcal H_{K\times K}(\omega_0)$  modulo  $Z(G)$, where
$$\mathcal H_{K\times K}(\omega_0)=\{\phi\in C^\infty(M)\ |~ \omega_\phi=\omega_0+\sqrt{-1}\partial\bar{\partial}\phi>0~{\rm and}~\phi ~\text{ ~ is~ $K\times K$-invariant }\}$$
and $Z(G)$ is the centre of $G$.
\end{theo}

\vskip 8pt

In case that  $M$ is Fano and $L=K_M^{-1}$,  then $\bar S=n$ and $\Lambda_A=1$ for all $A$. We have $\widetilde{bar}=bar$, thus \eqref{tildebar2}, \eqref{barS} are automatically satisfied. Moreover,   (\ref{bar}) is equivalent
   to  the vanishing of Futaki invariant and (\ref{tildebar1}) (cf.  Corollary \ref{+5205}).
Consequently, $\mu_{\omega_0}(\cdot)$ is proper modulo the action of $Z(G)$. Hence we get the an alternative proof
for the sufficient part of Theorem \ref{de} \cite{CT, TZ}.

As mentioned above,   we prove  Theorem \ref{LZZ1} by using the reduced K-energy $\mathcal K(\cdot)$.   One of the advantages  of $\mathcal K(\cdot)$ is that   it can be defined on a complete space  $\tilde{\mathcal C}_*$   of convex functions on $2P_+$.  Following the argument in \cite{ZZ08},   we discuss  the  semi-continuity property of $\mathcal K(\cdot)$. As  a consequence, we prove the following

\begin{theo}\label{LZZ2}
$\mathcal K(u)$ is lower semi-continuous on $\tilde{\mathcal C}_*$. Furthermore, if $\mu_{\omega_0}(\cdot)$ is proper  on   $\mathcal H_{K\times K}(\omega_0)$ modulo $Z(G)$,   then there exists a minimizer of $\mathcal K(\cdot)$ on $\tilde{\mathcal C}_*$.
\end{theo}

It is interesting to study the regularity of  minimizers in Theorem \ref{LZZ2}. We guess that they are smooth in $2P_+$ if the dimension of the torus $T^\mathbb C$ is  less than  two. In case of toric surfaces,  it is verified in \cite{Z1, Z2}.

The paper is organized as following: In Section 2, we review some preliminaries on $K\times K$-invariant metrics on $M$, and then we give a formula of scalar curvature  of  such metrics  in terms of  Legendre functions.  The formula of  $\mathcal K(\cdot)$  is obtained in Section 3. In  Section 4, we  use the idea   in \cite{ZZ}  for  toric manifolds to  prove Theorem \ref{LZZ1},  but there are new difficulties arising  from energy estimates near the Weyl walls to overcome.  In  Section 5,  we focus on the Fano case, and  prove  the properness of modified K-energy provided a modified  barycentre condition (\ref{bar4}) (cf. Theorem \ref{thm40}).  In Section 6, we prove  Theorem \ref{LZZ2}.

\section{Preliminaries}

In this section,    we  first recall  some preliminaries for  $K\times K$-invariant K\"ahler metrics on a  polarized compactification $(M, L)$ of   $G$
\cite{Del1, Del2, Del3} and the associated  Legendre functions,  then we  give a computation of scalar curvature in terms of  Legendre functions.

\subsection{Polarized compactification}

Let $J$ be the complex structure of $G$ and $K$ be one of its maximal compact subgroup such that $G=K^{\mathbb C}$.
Choose $T$ a maximal torus of $K$. Denote by $\mathfrak g$, $\mathfrak k$,
$\mathfrak t$ the corresponding Lie algebra of $G,\,K,\,T$, respectively. Then $$\mathfrak g=\mathfrak k\oplus J\mathfrak k.$$
Set $\mathfrak a=J\mathfrak t$ and Lie algebra of $Z(G)$ by  $\mathfrak z(\mathfrak g)$.   We decompose $\mathfrak a$ as a toric part and a semisimple part:
$$\mathfrak a = \mathfrak a_t\oplus\mathfrak a_{ss},$$
where $\mathfrak a_t:= \mathfrak z(\mathfrak g)\cap\mathfrak a$ and $\mathfrak a_{ss}:= \mathfrak a\cap[\mathfrak g,\mathfrak g]$.   Then for any $x\in\mathfrak a$,  we have  $x=x_t+x_{ss}$ with $x_t\in\mathfrak a_t$ and   $x_{ss}\in\mathfrak a_{ss}$. We extend the Killing form on $\mathfrak a_{ss}$ to a scalar product $\langle\cdot,\cdot\rangle$ on $\mathfrak a$ such that $\mathfrak a_t$ is orthogonal to $\mathfrak a_{ss}$. Identify $\mathfrak a$ and  its dual $\mathfrak a^*$ by $\langle\cdot,\cdot\rangle$.  Then $\mathfrak a^*$ also has an orthogonal decomposition
$$\mathfrak a^*=\mathfrak a_t^*\oplus\mathfrak a_{ss}^*.$$

Denote by $\Phi$ and $W$ the root system and Weyl group with respect to $(G,T^\mathbb C)$, respectively. Choose a system of positive roots $\Phi_+$.  Then it defines a positive Weyl chamber $\mathfrak a_+\subset\mathfrak a$, and a positive Weyl chamber ${\mathfrak a}^*_+$ on $\mathfrak a^*$,  where
 $${\mathfrak a}^*_+:=\{y|~\alpha(y):=\langle\alpha,y\rangle>0,~\forall\alpha\in\Phi_+\},$$
 which is also called the relative interior  $\Xi$ of the cone generated by $\Phi_+$.
The Weyl wall $W_\alpha$ is defined by  $W_\alpha:=\{y|~\alpha(y)=0\}$ for each $\alpha\in \Phi_+$.

\subsection{$K\times K$-invariant K\"ahler metrics}Let  $Z$ be the closure of $T^\mathbb C$ in $M$. It is known that $(Z, L|_Z)$ is a polarized toric manifold with a $W$-action, and $L|_Z$ is a $W$-linearized ample toric line bundle on $Z$  \cite{AB1,AB2, AK, Del2}. Let $\omega_0\in 2\pi c_1(L)$ be  a $K\times K$-invariant K\"ahler form induced from $(M, L)$  and  $P$ be the polytope associated to $(Z, L|_Z)$, which is defined by the moment map associated to $\omega_0$.   Then $P$ is a $W$-invariant  delzent polytope in $\mathfrak a^*$.
By the $K\times K$-invariance,  for any $\phi\in\mathcal H_{K\times K}(\omega_0)$, the restriction of $\omega_\phi$ on $Z$ is a toric K\"ahler metric. It induces a smooth strictly convex function $\psi$ on ${\mathfrak a}$, which is $W$-invariant \cite{AL}.

 By the $KAK$-decomposition (\cite{Kna}, Theorem 7.39), for any $g\in G$,
there are $k_1,\,k_2\in K$ and $x\in\mathfrak a$ such that $g=k_1\exp(x)k_2$. Here $x$ is uniquely determined up to a $W$-action. This means that $x$ is unique in $\bar{\mathfrak a}_+$.
Then we define a smooth $K\times K$-invariant function  $\Psi$   on $G$   by
$$\Psi( \exp(\cdot))=\psi(\cdot):~{\mathfrak a}\to\mathbb R.$$
Clearly $\Psi$ is well-defined since $\psi$ is $W$-invariant. We usually call $\psi$ \emph{the function associated to $\Psi$}. It can be verified  that $\Psi$ is a K\"ahler potential on $G$ such that $\omega=\sqrt{-1}\partial\bar\partial \Psi$ on $G$  (cf.   Lemma \ref{Hessian} below).

The following $KAK$-integral formula can be found in \cite{Kna2}, Proposition 5.28 (see also \cite{HY})

\begin{prop}\label{KAK int}
Let $dV_G$ be a Haar measure on $G$ and $dx$  the Lebesgue measure
on $\mathfrak{a}$.
Then there exists a constant $C_H>0$ such that for any
$K\times K$-invariant, $dV_G$-integrable function $\Psi$ on $G$,
$$\int_G \Psi(g)\,dV_G= C_H\int_{\mathfrak{a}_+}\mathbf{J}(x)\psi(x)\,dx,$$
where
$\mathbf J(x)=\prod_{\alpha \in \Phi_+} \sinh^2(\alpha(x)).$
\end{prop}

Next we recall the local holomorphic coordinates on $G$ used in  \cite{Del2}.   By the standard Cartan decomposition, we can decompose $\mathfrak g$ as
$$\mathfrak g=\left(\mathfrak t\oplus\mathfrak a\right)\oplus\left(\oplus_{\alpha\in\Phi}V_{\alpha}\right),$$
where $V_{\alpha}=\{X\in\mathfrak g|~ad_H(X)=\alpha(H)X,~\forall H\in\mathfrak t\oplus\mathfrak a\}$, the root space of complex dimension $1$ with respect to $\alpha$. By \cite{Hel}, one can choose $X_{\alpha}\in V_{\alpha}$ such that $X_{-\alpha}=-\iota(X_{\alpha})$ and
$[X_{\alpha},X_{-\alpha}]=\alpha^{\vee},$ where $\iota$ is the Cartan involution and $\alpha^{\vee}$ is a dual of $\alpha$ by the Killing form.
Let $E_{\alpha}:=X_{\alpha}-X_{-\alpha}$ and $E_{-\alpha}:=J(X_{\alpha}+X_{-\alpha})$. Denote by $\mathfrak k_{\alpha},\,\mathfrak k_{-\alpha}$ the real line spanned by $E_\alpha,\,E_{-\alpha}$, respectively.
Then we have the Cartan decomposition of $\mathfrak k$,
$$\mathfrak k=\mathfrak t\oplus\left(\oplus_{\alpha\in\Phi_+}\left(\mathfrak k_{\alpha}\oplus\mathfrak k_{-\alpha}\right)\right).$$
Choose a real basis $\{E^0_1,...,E^0_r\}$ of $\mathfrak t$.  Then $\{E^0_1,...,E^0_r\}$ together with $\{E_{\alpha},E_{-\alpha}\}_{\alpha\in\Phi_+}$ forms a real basis of $\mathfrak k$, which is indexed by $\{E_1,...,E_n\}$. $\{E_1,...,E_n\}$ can also be regarded as a complex basis of $\mathfrak g$. For any $g\in G$, we define local coordinates $\{z_{(g)}^i\}_{i=1,...,n}$ on a neighborhood of $g$ by
$$(z_{(g)}^i)\to\exp(z_{(g)}^iE_i)g.$$
It is easy to see  that $\theta^i|_g=dz_{(g)}^i|_g$,  where $\theta^i$ is the  dual of $E_i$,  which is a right-invariant holomorphic  $1$-form.    Thus
$\displaystyle{\wedge_{i=1}^n\left(dz_{(g)}^i\wedge d\bar{z_{(g)}^i}\right)}|_g$ is  also a  right-invariant  $(n,n)$-form,  which defines a Haar measure  $dV_G$.

The complex Hessian of the $K\times K$-invariant function  $\Psi$ in the above local coordinates was computed by  Delcroix as follows  \cite[Theorem 1.2]{Del2}.

\begin{lem}\label{Hessian}
Let $\Psi$ be a $K\times K$ invariant function on $G$, and $\psi$ the associated function
on $\mathfrak{a}$. Let $\Phi_+=\{\alpha_{(1)},...,\alpha_{(\frac{n-r}{2})}\}$. Then  for $x\in \mathfrak{a}_+$,
the complex Hessian matrix of $\Psi$  in the  above coordinates is diagonal by blocks, and equals to
\begin{equation}\label{+21}
\mathrm{Hess}_{\mathbb{C}}(\Psi)(\exp(x)) =
\begin{pmatrix}
\frac{1}{4}\mathrm{Hess}_{\mathbb{R}}(\psi)(x)& 0 &  & & 0 \\
 0 & M_{\alpha_{(1)}}(x) & & & 0 \\
 0 & 0 & \ddots & & \vdots \\
\vdots & \vdots & & \ddots & 0\\
 0 & 0 &  & & M_{\alpha_{(\frac{n-r}{2})}}(x)\\
\end{pmatrix},
\end{equation}
where
\[
M_{\alpha_{(i)}}(x) = \frac{1}{2}\langle\alpha_{(i)},\nabla \psi(x)\rangle
\begin{pmatrix}
\coth\alpha_{(i)}(x) & \sqrt{-1} \\
-\sqrt{-1} & \coth\alpha_{(i)}(x) \\
\end{pmatrix}.
\]
\end{lem}

By (\ref{+21}) in   Lemma \ref{Hessian},
we see that  $\psi$ is convex on $\mathfrak{a}$.
The complex Monge-Amp\'ere measure is given by $\omega^n_\phi=(\sqrt{-1}\partial\bar{\partial}\Psi)^n=MA_{\mathbb C}(\Psi)\,dV_G$, where
\begin{equation}\label{MA}
\mathrm{MA}_{\mathbb{C}}(\Psi)(\exp(x))= \frac{1}{4^{r+p}}
\mathrm{MA}_{\mathbb{R}}(\psi)(x)\frac{1}{\mathbf J(x)}\prod_{\alpha \in \Phi_+}\langle\alpha,\nabla \psi(x)\rangle^2.
\end{equation}

\subsection{Legendre functions}
By the  convexity  of $\psi$ on $\mathfrak a$,   the gradient $\nabla \psi$  defines a diffeomorphism from $\mathfrak{a}$ to the interior of the dilated polytope $2P$\footnote{We remark that the moment map is given by ${\frac12}\nabla \psi$, whose image is $P$.}.
Let $P_+:=P\cap\bar{\mathfrak a}_+^*$, then by the $W$-invariance of $\psi$ and $P$, the restriction of $\nabla \psi$ to $\mathfrak{a}_+$ is a diffeomorphism to the interior of $2P_+$.  We note that one part of $\partial (2P_+)$ lies on $\partial (2P)$ (which we call "outer faces") and the other part lies on Weyl walls $\{W_\alpha\}$.
 For simplicity, we may assume that $2P$ contains the origin $O$ in its interior. Then $2P$ can be described as the intersection of
$$l_{\tilde A}(y):=-u_{\tilde A}^iy_i+\lambda_{\tilde A}>0,\,\tilde A=1,...,d, $$
where $\lambda_{\tilde A}>0$ and $u_{\tilde A}$ are primitive vectors in $\mathfrak N$.

Recall that Guillemin's function of $2P$ is given by
\begin{equation}\label{Guillemin func}
u_0={\frac 1 2}\sum_{\tilde A}l_{\tilde A}(y)\log l_{\tilde A}(y).
\end{equation}
Set
\begin{equation}\label{L}
\begin{aligned}
\mathcal C_{\infty,W}=\{v|~v \text{ is strictly convex, } v-u_0\in C^{\infty}(\overline{2P}) \text{ and } v \text{ is } W\text{-invariant}\}
\end{aligned}\nonumber
\end{equation}
and
\begin{equation}
\mathcal C_{\infty,+}=\{v_{|_{2P_+}}|~v\in \mathcal C_{\infty,W}\}.\nonumber
\end{equation}
By \cite{Gui},  the Legendre function $u$ of $\psi$ belongs to $\mathcal C_{\infty,W}$.
The inverse is also true.   This means that any $u\in \mathcal C_{\infty,W}$ corresponds to a
K\"ahler potential in $\mathcal H_{K\times K}(\omega_0)$  (cf. \cite[Proposition 3.2]{AK}).

By a direct computation, we have
\begin{equation}\label{D of Guillemin}u_{0,i}=-{\frac12}\sum_{\tilde A}(\log l_{\tilde A}(y)+1)u_{\tilde A}^i,~u_{0,ij}={\frac12}\sum_{\tilde A}{\frac{u_{\tilde A}^iu_{\tilde A}^j}{l_{\tilde A}(y)}}.\end{equation}
Note that $u_0^{ij}\nu_i \to 0$ as $y\to F_{\tilde A}$, where  $(u_0^{ij})=(u_{0,ij})^{-1}$ and $\nu_{\tilde A}=(\nu_1,...,\nu_r)$ is the unit normal vector of face $F_{\tilde A}=\{y|~l_{\tilde A}(y)=0\}$. Similarly $-u^{ij}_{0,j}\nu_i \to \frac{2}{\lambda_{\tilde A}}\langle y, \nu_{\tilde A}\rangle$, where $u^{ij}_{0,k}={\frac{\partial u_0^{ij}}{\partial y_k}}$. Thus we get

\begin{lem}\label{bound-measure}
If $u\in \mathcal C_{\infty,W}$, then for any $\tilde A$, as $y\to F_{\tilde A}$,
\begin{eqnarray}
u^{ij}\nu_i \to 0~\text{ and }~u^{ij}_{,j}\nu_i \to \frac{2}{\lambda_{\tilde A}}\langle y, \nu_{\tilde A}\rangle,
\end{eqnarray}
where $(u^{ij})=(u_{,ij})^{-1}$ and $u^{ij}_{,k}={\frac{\partial u^{ij}}{\partial y_k}}$.
\end{lem}

\subsection{The scalar curvature}
We compute the Ricci curvature of $\omega_\phi$. Clearly it is also $K\times K$-invariant.
As in Lemma \ref{Hessian}, in the local coordinates  in Sect. 2.2,  ${\rm Ric}(\omega_{\phi})$ can be expressed as
\begin{align}\label{5102}
& -\mathrm{Hess}_{\mathbb{C}}(\log\det(\partial\bar{\partial}\Psi))(\exp(x))\notag\\
& =-
\begin{pmatrix}
\frac{1}{4}\mathrm{Hess}_{\mathbb{R}}(\tilde {\psi})(x)& 0 &  & & 0 \\
 0 & \tilde M_{\alpha_{(1)}}(x) & & & 0 \\
 0 & 0 & \ddots & & \vdots \\
\vdots & \vdots & & \ddots & 0\\
 0 & 0 &  & & \tilde M_{\alpha_{(\frac{n-r}{2})}}(x)\\
\end{pmatrix}\nonumber
\end{align}
for any $x\in\mathfrak a_+$, where
\begin{eqnarray}
\tilde {\psi}&=&\log\det(\nabla^2\psi)+2\sum_{\alpha\in\Phi_+}\log\alpha(\nabla \psi)+\chi(x),\nonumber\\
\chi(x)&=&-\log \mathbf J(x)=-2\sum_{\alpha\in\Phi_+}\log\sinh\alpha(x),\nonumber\\
\tilde M_{\alpha}(x) &=& \frac{1}{2}\left < \alpha , \nabla \tilde {\psi}\right>
\begin{pmatrix}
\coth\alpha( x) & \sqrt{-1} \\
-\sqrt{-1} & \coth \alpha (x) \\
\end{pmatrix}.\nonumber
\end{eqnarray}
Then the scalar curvature
\begin{equation}\label{+S1}
S(\omega_\phi)|_{\exp(x)}={\rm tr}\left(\nabla^2\psi)^{-1}\nabla^2{\tilde\psi}\right)+\sum_{\alpha\in\Phi_+}{\frac{\langle\alpha,\nabla{\tilde\psi}\rangle}{\langle\alpha,\nabla\psi\rangle}}.
\end{equation}

By using the Legendre function $u$, we get
\begin{lem}
\begin{align}\label{5103}
S(\omega_{\phi})=&-\sum_{i, j}\left(u^{ij}_{,ij}+4\sum_{\alpha\in\Phi_+}{\frac{\alpha_iu^{ij}_{,j}}{\alpha(y)}}+4\sum_{\alpha,\beta\in\Phi_+}{\frac{\alpha_i\beta_ju^{ij}}{\alpha(y)\beta(y)}}-2\sum_{\alpha\in\Phi_+}{\frac{\alpha_i\alpha_ju^{ij}}{(\alpha(y))^2}}\right)\notag\\
& -\sum_{i, k} u_{,ik}\left.{\frac{\partial^2 \chi}{\partial x^i\partial x^k}}\right|_{x=\nabla{u}}-2\sum_{i}\sum_{\alpha\in\Phi_+}\left.{\frac{\partial \chi}{\partial x^i}}\right|_{x=\nabla{u}}{\frac{\alpha_i}{\alpha(y)}},
\end{align}
where $y\in2P_+$, $u^{ij}_{,kl}=\frac{\partial^2 u^{ij}}{\partial y_k\partial y_l}$ and $\alpha_i$ are the components of $\alpha$.
\end{lem}

\begin{proof}
By the relations
\begin{equation}\label{+S2}
\begin{aligned}
(\nabla^2u)^{-1}|_y=(\nabla^2\psi)|_{x=\nabla u},~
\left.{\frac{\partial^3 \psi}{\partial x^i\partial x^j\partial x^k}}\right|_x={\frac{\partial}{\partial x^i}}\left(u^{jk}|_{y=\nabla \psi}\right)=\left.u^{jk}_{,l}u^{li}\right|_{y=\nabla \psi},\nonumber
\end{aligned}
\end{equation}
we have
\begin{equation}
\begin{aligned}
\left.{\frac{\partial\tilde\psi}{\partial x^p}}\right|_{x=\nabla u}&=u_{,ij}u^{ij}_{,k}u^{kp}+2\sum_{\alpha\in\Phi_+}{\frac{\alpha_lu^{lp}}{\alpha(y)}}+\left.{\frac{\partial\chi}{\partial x^p}}\right|_{x=\nabla u},\notag\\
\left.{\frac{\partial^2\tilde\psi}{\partial x^p\partial x^q}}\right|_{x=\nabla u}&=(u_{,ij}u^{ij}_{,k}u^{kp})_{,s}u^{sq}+2\sum_{\alpha\in\Phi_+}\left({\frac{\alpha_lu^{lp}}{\alpha(y)}}\right)_{,s}u^{sq}+\left.{\frac{\partial^2\chi}{\partial x^p\partial x^q}}\right|_{x=\nabla u}.\notag
\end{aligned}
\end{equation}
Substituting them into (\ref{+S1}), we obtain (\ref{5103}) immediately.
\end{proof}

Note $\pi(y)=\prod_{\alpha\in\Phi_+}(\alpha(y))^2$. Since
\begin{align}\label{dpi}
\frac{\partial\pi }{\partial y_i}(y)&=2\pi(y)\sum_{\alpha\in\Phi_+}{\frac{\alpha_i}{\alpha(y)}},\notag\\
\frac{\partial^2\pi }{\partial y_i\partial y_j}(y)&=\pi(y)\left(4\sum_{\alpha,\beta\in\Phi_+ }{\frac{\alpha_i\beta_j}{\alpha(y)\beta(y)}}
-2\sum_{\alpha\in\Phi_+}{\frac{\alpha_i\alpha_j}{(\alpha(y))^2}}\right),
\end{align}
 we can rewrite $S$ as
\begin{align}\label{+S}
S(\omega_\phi)=&-u^{ij}_{,ij}-2u^{ij}_{,j}{\frac{\pi_{,i}}{\pi}}-u^{ij}{\frac{\pi_{,ij}}{\pi}} \notag\\
&-u_{,ik}
\left.{\frac{\partial^2 \chi}{\partial x^i\partial x^k}}\right|_{x=\nabla{u}}-\left.{\frac{\partial \chi}{\partial x^i}}\right|_{x=\nabla{u}}{\frac{\pi_{,i}}{\pi}}.
\end{align}
By Proposition \ref{KAK int}, it follows
\begin{eqnarray*}
\int_MS\omega_\phi^n=C_H\int_{\mathfrak a_+}S\det(\nabla^2\psi)\prod_{\alpha\in\Phi_+}\langle\alpha,\nabla\psi\rangle^2\,dx
=C_H\int_{2P_+}S\pi\,dy.
\end{eqnarray*}
Since $\pi\equiv 0$ on each $W_\alpha$,   by integration by parts on (\ref{+S}), we  get
\begin{align}
&\int_{2P_+}S\pi\,dy\notag\\
=&-\int_{\partial(2P_+)}u^{ij}_{,j}\nu_i\pi\,d\sigma_0-\int_{2P_+}u^{ij}_{,j}\pi_{,i}\,dy-\int_{2P_+}u^{ij}\pi_{,ij}dy\notag\\
&-\int_{2P_+}{\frac{\partial}{\partial y_i}}\left(\left.{\frac{\partial\chi}{\partial x^j}}\right|_{x=\nabla u}\right)\pi\, dy-\int_{2P_+}\left.{\frac{\partial\chi}{\partial x^j}}\right|_{x=\nabla u}\pi_{,i}\,dy\notag\\
=&\sum_{ A}\int_{F_{A}}\left({\frac2{\lambda_{ A}}}\langle y,\nu_{A}\rangle+4\langle\rho,\nu_{ A}\rangle\right)\pi\,d\sigma_0=\sum_{A}\Lambda_{ A}\int_{F_{ A}}\langle y,\nu_{A}\rangle\pi\,d\sigma_0\notag\\
=&n\sum_{A}\Lambda_{A}\int_{E_{A}}\pi\,dy.\notag
\end{align}
Here we used Lemma \ref{bound-measure} and the fact that ${\frac{\partial \chi}{\partial x^i}}(x)\to-4\rho_i$   as  $x\to\infty$.
On the other hand, by Proposition \ref{KAK int},   the volume of $(M,\omega_\phi)$ is given by
 \begin{eqnarray*}
V_M:=\int_M\omega_\phi^n
&=&C_H\int_{\mathfrak a_+}MA_{\mathbb R}(\psi)\prod_{\alpha\in\Phi_+}\langle\alpha,\nabla\psi\rangle^2dx\\
&=&C_H\int_{2P_+}\pi\,dy.
\end{eqnarray*}
Hence,  combining the above two relations,  we get (\ref{5304}).

\section{Reduction of the K-Energy}

Let $(M, L)$ and $\omega_0\in2\pi c_1(L)$ be as before. Denote by $\mathcal H(\omega_0)$ the space of K\"ahler potentials in $[\omega_0]$. Mabuchi's K-energy is defined on $\mathcal H(\omega_0)$ by
\begin{equation}\label{5101}
\mu_{\omega_0}(\phi)=-{\frac 1 {V_M}}\int_0^1\int_M\dot{\phi_t}(S(\omega_{\phi_t})-\bar{S})\omega_{\phi_t}^n\wedge dt,
\end{equation}
where $V_M=\int_M\omega_0^n$, $\bar S$ is the average  of $ S(\omega_0)$ and $\{\phi_t\}$ is a path of K\"ahler potentials joining $0$ and $\phi$ in $\mathcal H(\omega_0)$.
In this section, we give a formula of $\mu_{\omega_0}(\cdot)$ on $\mathcal H_{K\times K}(\omega_0)$ in terms of the Legendre function $u$.

\subsection{Reduced K-energy}
Define
\begin{equation}\label{5105}
\begin{aligned}
\mathcal K(u)&=\sum_A\int_{F_A}\Lambda_A\langle y,\nu_A\rangle u\pi \,d\sigma_0-\int_{2P_+}\bar Su\pi \,dy\\
&\ \ \ -\int_{2P_+}\log\det(u_{ij})\pi \,dy+\int_{2P_+}\chi(\nabla u)\pi \,dy,
\end{aligned}\nonumber
\end{equation}
where   $\chi(x)=-\log \mathbf J(x)=-2\sum_{\alpha\in\Phi_+}\log\sinh\alpha(x)$ for any   $ x\in\mathfrak a.$    Then we have

\begin{prop}\label{5104}
Let $\phi\in\mathcal H_{K\times K}(\omega_0)$ and $u$ be the Legendre function of $\psi=\psi_0+\phi$.
Then
$$
\mu_{\omega_0}(\phi)={\frac{1}{V}}\mathcal K(u)+const.,
$$
where $V=\int_{2P_+}\pi\, dy$.
\end{prop}

\begin{proof}
Note $\dot{\phi}_t=-\dot{u}_{t}$.  By (\ref{MA}), it is easy to see
$${\frac 1 {C_H}}\int_0^1\int_M \bar{S}\dot{\phi_t} \omega_{\phi_t}^n\wedge dt=\int_{2P_+}\bar Su\pi \,dy.$$
Then by \eqref{+S}, it suffices to compute the part
\begin{eqnarray*} I&:=&-{\frac 1 {C_H}}\int_0^1\int_M\dot{\phi_t}S(\omega_{\phi_t})\,\omega_{\phi_t}^n\wedge dt\notag\\
&=&
-\int_0^1\int_{\mathfrak a^+}\dot{\phi_t}S(\omega_{\phi_t})|_{\exp(x)}MA_{\mathbb{R}}(\psi_{t})|_x\prod_{\alpha\in\Phi_+}\langle\alpha,\nabla \psi_{t}\rangle^2|_x \,dx\wedge dt.\notag\\
&=&\int_0^1\int_{2P_+}\dot{u}_{t}(-{u_{t}}^{ij}_{,ij})\pi \,dy\wedge dt\notag\\
&&
+\int_0^1\int_{2P_+}\dot{u}_{t}(-2{u_{t}}^{ij}_{,j}\pi_{,i}) \, dy\wedge dt
+\int_0^1\int_{2P_+}\dot{u}_{t}({-u_{t}^{ij}}\pi_{,ij}) \, dy\wedge dt\nonumber\\
&& +\int_0^1\int_{2P_+}\dot{u}_{t}\left(-{u_t}_{,ik}{\frac{\partial^2 \chi}{\partial x^i\partial x^k}}|_{x=\nabla{u_{t}}}\pi-\sum_{\alpha\in\Phi_+}{\frac{\partial \chi}{\partial x^i}}|_{x=\nabla{u_{t}}}\pi_{,i}\right) \, dy\wedge dt.
\end{eqnarray*}
By integration by parts,  it follows
\begin{eqnarray}\label{5108}
I&=&\int_0^1\int_{\partial(2P_+)}\dot{u}_{t}(-{u_{t}}^{ij}_{,j}\nu_i)\pi \,d\sigma_0\wedge dt
+\int_0^1\int_{2P_+}\dot{u}_{t,i}{u_{t}}^{ij}_{,j}\pi \,dy\wedge dt\nonumber\\
&&
-\int_0^1\int_{2P_+}\dot{u}_{t}({u_{t}}^{ij}_{,j}\pi_{,i}) \,dy\wedge dt
-\int_0^1\int_{2P_+}\dot{u}_{t}({u_{t}^{ij}}\pi_{,ij}) \,dy\wedge dt\nonumber\\
&&-\int_0^1\int_{2P_+}\dot{u}_{t}{\frac{\partial}{\partial y_i}}\left(\left.{\frac{\partial \chi}{\partial x^i}}\right|_{x=\nabla u_{t}}\pi\right) \, dy\wedge dt.
\end{eqnarray}

Note that ${\frac{\partial \chi}{\partial x^i}}(x)\to-4\rho_i$ as $x\to\infty$ in $\mathfrak a_+$ and is away from Weyl walls, and $\pi$ vanishes quadratically along any Weyl wall. Then the last term in \eqref{5108}  becomes
\begin{align}\label{5100}
&\left.\int_{2P_+}\chi(\nabla u_{t})\pi \, dy\right|_0^1-\int_0^1\int_{\partial(2P_+)}\dot u_{t}\left.{\frac{\partial \chi}{\partial x^i}}\right|_{x=\nabla u_{t}}\nu_i\pi \,d\sigma_0\wedge dt\notag\\
&=\int_{2P_+}\chi(\nabla u)\pi \, dy+4\int_{\partial(2P_+)}\langle\rho,\nu\rangle u\pi\,d\sigma_0+const.
\end{align}
On the other hand, by the second relation in Lemma \ref{bound-measure}, we have
\begin{align}\label{+5103'}
&\int_0^1\int_{2P_+}\dot{u}_{t,i}{u_{t}}^{ij}_{,j}\pi \,dy\wedge dt-\int_0^1\int_{2P_+}\dot{u}_{t}({u_{t}}^{ij}_{,j}\pi_{,i}) \,dy\wedge dt\notag\\
&=\int_0^1\int_{\partial(2P_+)}\dot{u}_{t,i}{u_{t}}^{ij}\nu_j\pi \,d\sigma_0\wedge dt-\int_0^1\int_{2P_+}\dot{u}_{t,ij}{u_{t}}^{ij}\pi \,dy\wedge dt\notag\\
&\ \  +\int_0^1\int_{2P_+}\dot{u}_{t}{u_{t}}^{ij}\pi_{,ij} \,dy\wedge dt  -\int_0^1\int_{\partial(2P_+)}\dot{u}_{t}{u_{t}}^{ij}\nu_j\pi_{,i} \,d\sigma_0\wedge dt  \notag\\
&=  -\int_0^1\int_{2P_+}{\frac d{dt}}
\left[\log\det(u_{t,ij})\right]\pi\,dy\wedge dt +\int_0^1\int_{\partial(2P_+)}\dot{u}_{t}{u_{t}}^{ij}\pi_{,ij} \,d\sigma_0\wedge dt.
\end{align}
Thus combining (\ref{+5103'}) and (\ref{5100}),  we get from (\ref{5108}),
\begin{eqnarray*}
I&=&\int_0^1\int_{\partial(2P_+)}\dot{u}_{t}(-{u_{t}}^{ij}_{,j}\nu_i)\pi \,d\sigma_0\wedge dt+4\int_{\partial(2P_+)}\langle\rho,\nu\rangle u\pi\,d\sigma_0\\&&-\int_{2P_+}\log\det(u_{,ij})\pi \,dy
+\int_{2P_+}\chi(\nabla u)\pi \,dy+const.
\end{eqnarray*}
By Lemma \ref{bound-measure},  we see
$$
\int_{\partial(2P_+)}\dot{u}_{t}(-{u_{t}}^{ij}_{,j}\nu_j)\pi\, d\sigma_0=\sum_A\int_{F_A}\dot{u}_{t}{\frac{2}{\lambda_A}}\langle y,\nu_A\rangle\pi \,d\sigma_0.
$$
Hence, we obtain
$$
I=\sum_A\int_{F_A}\Lambda_A\langle y,\nu_A\rangle u\pi \,d\sigma_0-\int_{2P_+}[\log\det(u_{ij})-\chi(\nabla u)]\pi \,dy+const.$$
Recall that $V_M=C_H\cdot V$, the proof is finished.
\end{proof}

\vskip 7pt

For convenience, we write $\mathcal K(u)$ as
$\mathcal K(u)=\mathcal L(u)+\mathcal N(u)$, where
\begin{align} \label{L}
\mathcal L(u)&=\sum_A\int_{F_A}\Lambda_A\langle y,\nu_A\rangle u\pi \,d\sigma_0-\int_{2P_+}\bar Su\pi\, dy-\int_{2P_+}4\langle\rho,\nabla u\rangle\pi\,dy,
\end{align}\label{N}
\begin{align}\mathcal N(u)&=-\int_{2P_+}\log\det\left(u_{,ij}\right)\pi\,dy+\int_{2P_+}[\chi\left(\nabla u\right)+4\langle\rho,\nabla u\rangle ]\pi\,dy.
\end{align}
By integration by parts, we can rewrite $\mathcal L(u)$ as
\begin{equation}\label{5119}
\mathcal L(u)=\sum_A\int_{ E_A}\left[\langle \Lambda_A y-4\rho, \nabla u\rangle+(\Lambda_An-\bar S)u\right]\pi \,dy,
\end{equation}
or
\begin{equation}\label{+5119}
\mathcal L(u)=\sum_A{\frac2{\lambda_A}}\int_{ F_A}\langle y, \nu_A\rangle u\pi \,d\sigma_0-\int_{2P_+}\bar Su\pi \,dy+\int_{2P_+}4\langle\rho,\nabla \pi\rangle u \,dy.
\end{equation}
\subsection{The Futaki Invariant}
In this subsection, we discuss the relationship between the  Futaki invariant $F(\cdot)$ and the linear part $\mathcal L(\cdot)$ of  $\mathcal K(\cdot)$.

Let $\text{Aut}^0(M)$ be  the identity component of the automorphisms group of $M$ with Lie algebra $\eta(M)$.  Let  $\text{Aut}_r(M)$ be a reductive algebraic subgroup of $\text{Aut}^0(M)$.  Then $\text{Aut}_r(M)$ is the  complexification
of a maximal compact subgroup $K_r$ (with Lie algebra $\mathfrak k_r$).
Denote the Lie algebra of $\text{Aut}_r(M)$ by $\eta_r(M)$ and its centre by $ \eta_c(M)$.
By a result of Futaki \cite{Fu}, it suffices to consider $F(v)$ for  holomorphic vector fields $v\in\eta_c(M)$. In our case $(M,L)$,   when   $v$ is restricted on $G$,  $v=\sqrt{-1}v^iE_i^0$ with $\alpha(v)=0$, for any $\alpha\in\Phi$,
where  $v^i\in\mathbb C$ are  some constants, $i=1,...,r$.   If  ${\rm Im}(v)\in\mathfrak k_r$, $v^i$ are all real numbers. In particular, Re$(v)\in\mathfrak a_t$.

\begin{lem}\label{5205}Let $l_v(y)=\sum_i v^iy_i$ be the linear function associated to $v\in\eta_c(M)$. Then
the  Futaki invariant is given by
\begin{align}\label{futaki}
F(v)={\frac1V}\mathcal L(l_v(y)).
\end{align}
\end{lem}

\begin{proof}Let $\sigma^v_t$ be the one-parameter group generated by Re$(v)$ and $\phi_t^v$ be a family of induced  K\"ahler potentials by
$${\sigma_t^v}^*\omega_0=\omega_0+\sqrt{-1}\partial\bar{\partial}\phi_t^v.$$
Since
$\sigma^v_t(k_1\exp(a)k_2)=k_1\exp(a+tv)k_2$ for any $a\in \mathfrak a,$
${\sigma_t^v}^*\omega_0$ is $K\times K$ invariant.
Then $\{{\sigma_t^v}^*\omega_0\}$ induces a family of $W$-invariant convex functions $\{\psi_t\}$ on $\mathfrak a$. Moreover, the Legendre functions  $u_t$  of  $\psi_t$ are given by
$$u_t=u_0-tl_v(y).$$
By Proposition \ref{5104},  we get
\begin{equation}\label{5203}
F(v)=-{\frac1V}{\frac{d}{dt}}\mu_{\omega_0}(\phi_t)=-{\frac1V}\mathcal L(l_v(y))-{\frac1V}\int_{2P_+}\left(v^i\left.{\frac{\partial \chi}{\partial x^i}}\right|_{x=\nabla{u_{\phi_t}}}+4\rho_iv^i\right)\pi \,dy.\nonumber
\end{equation}
Note that $\alpha(v)=0$ for all $\alpha\in\Phi$, which implies $\rho_iv^i=0$ and
$$v^i\frac{\partial \chi}{\partial x^i}(x)=-2\sum_{\alpha\in\Phi_+}\alpha(v)\coth\alpha(x)=0.$$
  Hence (\ref{futaki}) is true.
\end{proof}

\begin{cor}\label{+5205}
$M$ has vanishing Futaki invariant  if and only if
$\mathcal L(l_v)=0$ for any $v\in\eta_c(M)$.  The later is equivalent to
\begin{equation}\label{5206}
\begin{aligned}
&\left(\widetilde{bar}-{\frac{n}{n+1}}bar\right)\in\mathfrak a_{ss}.\nonumber
\end{aligned}
\end{equation}
\end{cor}
\begin{proof}
By (\ref{5119}) and (\ref{5304}), we have $$
\begin{aligned}
\mathcal L(l_v)&=(n+1)\sum_A\int_{E_A}\Lambda_Ay_iv^i\pi \,dy-\int_{2P_+}\bar Sy_iv^i\pi\, dy\\
&=(n+1)\cdot\left(\sum_A\Lambda_A\int_{E_A}\pi\, dy\right)\cdot\left\langle \widetilde{bar}-{\frac{n}{n+1}}bar,v\right\rangle.
\end{aligned}$$
This proves the corollary.
\end{proof}

Another explanation of $\mathcal L(u)$ for a $W$-invariant, convex piecewise linear $u$ can be described as the generalized Futaki-invariant corresponding to a toric degeneration $\mathfrak U$ as done in \cite[Theorem 3.3]{AK}.  In fact,
\begin{align}\label{generalized-futaki} F(\mathfrak U)={\frac{1}{2\int_{P_+}H_rdy}}\left(\int_{\partial P_+}uH_rd\sigma+2\int_{P_+}uH_{r-1}dy-a\int_{P_+}uH_rdy\right),
\end{align}
where $a={\frac{\int_{\partial P_+}H_rd\sigma+2\int_{P_+}H_{r-1}dy}{\int_{P_+}H_rdy}}$.
The coefficients $H_{*}$ arise from the homogeneous expression
$$\dim(End (E_{\varpi}))=H_r(\varpi)+H_{r-1}(\varpi)+...,~\varpi\in\mathfrak a^*_+\cap\mathfrak M,$$ for the irreducible $G$-representation $E_\varpi$ of highest weight $\varpi$. By the Weyl character formula, $H_{r}(y)={\frac{\pi(y)}{\prod_{\alpha\in\Phi_+}\langle\alpha,\rho\rangle^2}}$ and $H_{r-1}(y)={\frac{\langle\nabla\pi(y),\rho\rangle}{\prod_{\alpha\in\Phi_+}\langle\alpha,\rho\rangle^2}}$.
Thus by changing the integral variable $y$ to ${\frac12} y$ in (\ref{generalized-futaki}), we see that $a=\bar S$ and \begin{equation}\label{+generalized-futaki}\mathcal L(u)=V\cdot F(\mathfrak U),\end{equation} for any $W$-invariant rational convex piecewise linear $u$.

In Fano case,  we have  all $\Lambda_A=1$. This is because there is a smooth $K\times K$-invariant Ricci potential $H_0$ on $M$ so that
\begin{equation}\label{5114}
-\log\det(\partial\bar{\partial}\Psi_0)-\Psi_0=H_0.\nonumber
\end{equation}
Then it reduces to a bounded smooth $h_0$ on $\mathfrak a$,
\begin{equation}\label{5115}
\begin{aligned}
-h_0&=\log\det(\psi_{0,ij})+\psi_0-\log \mathbf J(x)\\
&=-\log\det(u_{0,ij})+y_iu_{0,i}-u_0+\chi(\nabla u_0).
\end{aligned}
\end{equation}
By (\ref{D of Guillemin}), the singular terms on the right hand side for $y\in2P_+$ is
\begin{equation}\label{5116}
\sum_{ A}\left(1-{\frac1 2}y_iu_{A}^i+2\rho_iu_{A}^i\right)\log l_{A}(y).\nonumber
\end{equation}
It follows $$\lambda_{A}(\Lambda_{A}-1)=2-u_{A}^iy_i+4\rho_iu_{A}^i=0.$$
Thus $\Lambda_A=1$.

Now, in Fano case, we see that $\widetilde{bar}=bar$.  Then (\ref{bar}) implies that $bar\in\mathfrak a_{ss}$.   By Corollary \ref{+5205}, the Futaki invariant vanishes.  Furthermore,    by (\ref{5119}),
we get
\begin{equation}\label{L-fano}
\mathcal L(u)=\int_{2P_+}\langle y-4\rho,\nabla u\rangle\pi\,dy.
\end{equation}
The following proposition shows that (\ref{bar}) is a necessary condition of the existence of K\"ahler-Einstein metrics on $(M,L)$ from the view of K-stability.

\begin{prop}\label{necessary}
Let $(M,L)$ be a Fano compactification of $G$. Then $M$ is not K-stable if $bar-4\rho\not\in\Xi$.
\end{prop}
\begin{proof}
Let $\{\alpha_{(1)},...,\alpha_{(r')}\}$ be the simple roots in $\Phi_+$. Since $bar-4\rho\not\in\Xi$, without loss of generality we can write
$$bar-4\rho=\lambda_1\alpha_{(1)}+...+\lambda_{r'}\alpha_{(r')}+v,$$
where $\lambda_1\leq0$ and $v\in\mathfrak a^*_{t}$. Let $\{\varpi_i\}$ be the fundamental weights for $\{\alpha_{(1)},...,\alpha_{(r')}\}$ such that ${\frac{2\langle \varpi_i,\alpha_{(j)}\rangle}{|\alpha_{(j)}|^2}}=\delta_{ij}$. Define a $W$-invariant rational piecewise linear function $u$ on $2P$ by
$$u(y)=\max_{w\in W}\{\langle w\cdot \varpi_1,y\rangle\}.$$
Then $u$ defines a  non-trivial   toric  degeneration. Since $\varpi_1$ is dominant, we have
$$u|_{2P_+}(y)=\langle \varpi_1,y\rangle.$$
 Thus by (\ref{L-fano}), we get
$$\mathcal L(u)=\langle bar-4\rho,\varpi_1\rangle={\frac12}|\alpha_{(1)}|^2\lambda_1\leq0.$$
By (\ref{+generalized-futaki}), the proposition is proved.
\end{proof}

\section{A criterion for properness of the K-Energy}

In this section, we study the properness of the K-energy  associated to a general K\"ahler class  $\omega_0$. We reduce the problem to $\mathcal K(\cdot)$.

Let $O$ be the origin of $\mathfrak a^*$. Note that $\mathfrak a^*_t$ is the fixed point set of the $W$-action. Then $\nabla u(O)\in\mathfrak a^*_t$ for any $u\in\mathcal C_{\infty,W}$. We can normalize $u\in\mathcal C_{\infty,W}$ by
\begin{equation}\label{tilde u}\tilde u(y)=u(y)-\langle\nabla u(O),y\rangle-u(O).\end{equation}
Then $\tilde u\in\mathcal C_{\infty,W}$ and
\begin{align}\label{normalization-u}\min_{2P}\tilde u=\tilde u(O)=0.
\end{align}
The subset of normalized functions in $\mathcal C_{\infty,W}$ and $\mathcal C_{\infty,+}$ will be denoted by $\hat {\mathcal C}_{\infty,W}$ and $\hat {\mathcal C}_{\infty,+}$, respectively.
The following proposition gives a criterion for the properness of $\mathcal K(\cdot)$.

\begin{prop}\label{proper-red}
Under the assumption of Theorem \ref{LZZ1},
for any $\delta\in(0,1)$, there exists a uniform constant $C_{\delta}>0$, such that
\begin{equation}\label{608}
\mathcal K(u)\geq \delta\int_{2P_+}u\pi \,dy-C_{\delta}, \, \forall \ u\in\hat{\mathcal C}_{\infty,+}.
\end{equation}
\end{prop}

We shall estimate both of  the linear part $\mathcal L(\cdot)$ and nonlinear part $\mathcal N(\cdot)$ of   $\mathcal K(\cdot) $ below.

\subsection{Estimate of $\mathcal L(\cdot)$}
The following lemma can be directly proved from the convexity of $u$.

\begin{lem}\label{norm-convex}
There is a uniform constant $\Lambda$, such that
\begin{equation}\label{44}
\int_{2P_+}u\pi\,dy\leq\Lambda\int_{\partial(2P_+)}u\langle y,\nu\rangle\pi\,d\sigma_0,\ \forall u\in\hat {\mathcal C}_{\infty,+}.\nonumber
\end{equation}
\end{lem}

Now we prove
\begin{prop}\label{5301}
Under the assumption of   Theorem \ref{LZZ1},  there exists a positive constant $\lambda$ such that
\begin{equation}\label{5302}
\mathcal L(u)\geq\lambda\int_{\partial(2P_+)}\langle y,\nu\rangle u\pi\,d\sigma_0,~\forall u\in\hat {\mathcal C}_{\infty,+}.\nonumber
\end{equation}
\end{prop}

\begin{proof}
Since $u$ is convex, we have $$\langle y-\widetilde{bar}_{ss},\nabla u(y)\rangle\geq u(y)-u(\widetilde{bar}_{ss}).$$
By (\ref{5119}), we have
\begin{equation}\begin{aligned}
&\ \ \ \mathcal L(u)\notag\\
&=\sum_A\int_{ E_A}\Lambda_A\langle y-\widetilde{bar}_{ss},\nabla u\rangle\pi  dy+\sum_A\int_{ E_A}\langle\Lambda_A\widetilde{bar}_{ss}-4\rho,\nabla u\rangle\pi dy\\
&\ \ \ +\sum_A\int_{ E_A}(\Lambda_An-\bar S)u\pi dy\\
&\geq\sum_A\int_{ E_A} (\Lambda_A(n+1)-\bar S ) [u(y)-u(\widetilde{bar}_{ss})-\langle\nabla{u}|_{\widetilde{bar}_{ss}},y-\widetilde{bar}_{ss}\rangle ]\pi dy\\
&\ \ \ +\sum_A\int_{ E_A}\langle\Lambda_A\widetilde{bar}_{ss}-4\rho,\nabla u\rangle\pi  dy
+\sum_A\int_{ E_A}\Lambda_A\langle\nabla{u}|_{\widetilde{bar}_{ss}},y-\widetilde{bar}_{ss}\rangle\pi dy\\
&\ \ \ +\sum_A\int_{ E_A}(\Lambda_An-\bar S)\left(\langle\nabla{u}|_{\widetilde{bar}_{ss}},y-\widetilde{bar}_{ss}\rangle+u(\widetilde{bar}_{ss})\right)\pi dy.\notag
\end{aligned}\end{equation}
By \eqref{5304}, the last two terms equals
\begin{equation}\label{5305}
[\langle(n+1)\widetilde{bar}_t-n\cdot bar_t,\nabla u|_{\widetilde{bar}_{ss}}\rangle
+n\langle\widetilde{bar}_{ss}-bar_{ss},\nabla{u}|_{\widetilde{bar}_{ss}}\rangle] \sum_A\int_{ E_A}\Lambda_A\pi  dy.
\nonumber\end{equation}
Note that $\mathfrak a_t$ is orthogonal to $\mathfrak a_{ss}$. Choosing Re$(v)=(\nabla u|_{\widetilde{bar}_{ss}})_t$ in Corollary \ref{+5205}, we have $$\langle(n+1)\widetilde{bar}_t-n\cdot bar_t,\nabla{u}|_{\widetilde{bar}_{ss}}\rangle=\langle(n+1)\widetilde{bar}_t-n\cdot bar_t,(\nabla u|_{\widetilde{bar}_{ss}})_t\rangle=0.$$
Thus
\begin{align}\label{lower bounde of L}
&\ \ \ \mathcal L(u)\notag\\
 &\geq \sum_A\int_{ E_A}(\Lambda_A(n+1)-\bar S ) [u(y)-u(\widetilde{bar}_{ss})-\langle\nabla{u}|_{\widetilde{bar}_{ss}},y-\widetilde{bar}_{ss}\rangle ]\pi(y) dy\nonumber\\
&\ \ \ +\sum_A\int_{ E_A}\langle\Lambda_A\widetilde{bar}_{ss}-4\rho,\nabla u\rangle\pi \,dy\notag\\
&\ \ \ +n\left(\sum_A\int_{ E_A}\Lambda_A\pi\, dy\right)\langle\widetilde{bar}_{ss}-bar_{ss},\nabla{u}|_{\widetilde{bar}_{ss}}\rangle.
\end{align}
Condition (\ref{tildebar2}) implies
$\langle\widetilde{bar}_{ss}-bar_{ss},\nabla{u}|_{\widetilde{bar}_{ss}}\rangle\geq 0,$
 while (\ref{tildebar1}) implies
$$\sum_A\int_{ E_A}\langle\Lambda_A\widetilde{bar}_{ss}-4\rho,\nabla u\rangle\pi \,dy\geq 0.$$
Moreover, each equality holds if and only if $\nabla u(y)\in\mathfrak a_t$ for all $y\in2P_+$. Hence the three terms in (\ref{lower bounde of L}) are all nonnegative for $u\in\hat{\mathcal C}_{\infty,+}$.

We want to use (\ref{lower bounde of L})  to  prove  the lemma.  Suppose that  it  is not true.  Then  there exists a sequence $\{u_k\}\subset\hat {\mathcal C}_{\infty,+}$ such that
\begin{equation}\label{323}
\int_{\partial(2P_+)}u_k\langle y,\nu\rangle\pi\,d\sigma_0=1
\text{ and }
\mathcal L(u_k)\to0,\,k\to\infty.
\end{equation}
Thus there is a subsequence (still denoted by $\{u_k\}$) which converges locally uniformly to a convex function $u_{\infty}$ in $2P_+$. Since the last two terms of (\ref{lower bounde of L}) is nonnegative, we have
\begin{equation}\label{+324}\begin{aligned}
0&\leq\sum_A\int_{ E_A}\left(\Lambda_A(n+1)-\bar S\right)\left[u_k(y)-u_k(\widetilde{bar}_{ss})-\langle\nabla{u}|_{\widetilde{bar}_{ss}},y-\widetilde{bar}_{ss}\rangle\right]\pi(y) dy\\&\leq\mathcal L(u_k)\to0.\end{aligned}\nonumber
\end{equation}
Hence  $u_{\infty}$ must be  an affine linear function. By the fact  $u_k(O)=0$, we have $u_{\infty}(O)=0$
 and so
$u_{\infty}=\xi^iy_i$
for some $\xi=(\xi^i) \in\bar{\mathfrak a}_+$.

Substituting $u_{\infty}$ into (\ref{L}), we have
\begin{equation}\label{+L}
\begin{aligned}
0&=\mathcal L(u_{\infty})\\&=\sum_A\int_{ E_A}\langle\Lambda_A\widetilde{bar}_{ss}-4\rho,\xi\rangle\pi\, dy+n\left(\sum_A\int_{ E_A}\Lambda_A\pi\, dy\right)\langle\widetilde{bar}_{ss}-bar_{ss},\xi\rangle\geq 0.
\end{aligned}\nonumber
\end{equation}
Note that $\langle \Lambda_A\widetilde{bar}_{ss}-4\rho,\xi\rangle\geq0$
and $\langle\widetilde{bar}_{ss}-bar_{ss},\xi\rangle\geq0$ with "=" holds
iff $\xi\in\mathfrak a_t$. By $\mathcal L(u_{\infty})=0$, we get $\xi\in\mathfrak a_{t}$.
This implies that $u_{\infty}(y)$ is a linear function depending only on $y_{t}$, i.e., the projection of  $y$ in $\mathfrak a^*_t$.
Since $O$ lies in the interior of $\mathfrak a_t^*\cap(2P_+)$ and $u_{\infty}\geq0$,
we get  $u_{\infty}=0$. As a consequence,
\begin{eqnarray}
\int_{2P_+}u_k\pi\,dy&
\to0,\text{ as $k\to\infty.$}\label{328}
\end{eqnarray}
On the other hand, since all $\lambda_A>0$,
there exists a uniform constant $\lambda_0>0$, such that for $u\in\hat{\mathcal C}_{\infty,+}$,
\begin{equation}\label{326}
 \sum_A{\frac2{\lambda_A}} \int_{F_A}\langle y,\nu_A\rangle u\pi\, d\sigma_0\geq2\lambda_0\int_{\partial(2P_+)}u\langle y,\nu\rangle\pi\,d\sigma_0.\nonumber
\end{equation}
Note that  $\langle4\rho,\nabla \pi\rangle\geq0$ on $2P_+$.
Hence, substituting (\ref{323}), (\ref{328}) and the above equality for $u=u_k$  into (\ref{+5119}), we see
 $\mathcal L(u_k)\geq\lambda_0>0$, which contradicts to the second relation  in  (\ref{323}). The lemma is proved.
\end{proof}

\subsection{Estimate of $\mathcal N$}
We prove

\begin{prop}\label{non-linear part}
There exist uniform constants $C_{\Lambda},\,C_L,\,C_0>0$ such that for any $u\in\hat{\mathcal C}_{\infty,+}$,
\begin{eqnarray}\label{+4N'}
\mathcal N(u)\geq -C_{\Lambda}\int_{\partial(2P_+)}u\langle y,\nu\rangle\pi\,d\sigma_0-C_L\mathcal L(u)+\int_{2P_+}Qu \pi\, dy -C_0,
\end{eqnarray}
where
\begin{equation}\label{59}
Q=-\left.{\frac{\partial \chi}{\partial x^i}}\right|_{x=\nabla u_0}{\frac{\pi_{,i}}{\pi}}-\left.{\frac{\partial^2 \chi}{\partial x^i\partial x^k}}\right|_{x=\nabla u_0}u_{0,ik}- u_0^{ij}{\frac{\pi_{,ij}}{\pi}}.\end{equation}
\end{prop}

\begin{proof}
First,  we note that $\chi(\cdot)$ is strictly convex on $\mathfrak a_+$ (cf. \cite[Lemma 3.7]{Del2}). Then by the convexity of $-\log\det$, we have
\begin{align}\label{+42}
&\ \ -\log\det(u_{,ij})+\chi(\nabla u)\notag\\
&\geq-\log\det(u_{0,ij})+\chi(\nabla u_0)-u_0^{ij}(u_{,ij}-u_{0,ij})+\left.{\frac{\partial\chi}{\partial x^i}}\right|_{x=\nabla u_0}(u_{,i}-{u_{0,i}}).
\end{align}
By (\ref{N}), it follows
\begin{align}\label{42}
&\ \ \  \mathcal N(u)\notag\\
&\geq-\int_{2P_+}u_0^{ij}u_{,ij}\pi\,dy
+\int_{2P_+}u_{,i}\left.{\frac{\partial \chi}{\partial x^i}}\right|_{x=\nabla u_0}\pi\,dy
+4\int_{2P_+}\langle\rho, \nabla u\rangle\pi\,dy-C_0
\end{align}
for some constant $C_0$ independent of $u$.
Since $\nabla \chi(x)+4\rho_i x^i$ vanishes at infinity away from Weyl walls and $\pi(y)$ vanishes quadratically along any Weyl wall,
\begin{equation}\label{410}
\int_{\partial(2P_+)}\left(\left.{\frac{\partial \chi}{\partial x^i}}\right|_{x=\nabla u_0}+4\rho_i\right)\nu^iu\pi\, d\sigma_0 =0.\nonumber
\end{equation}
Thus by integration by parts for the first integral terms in (\ref{42}), and then by Lemma \ref{bound-measure}, we get
\begin{align}\label{49}
\mathcal N(u)&\geq-\sum_A\int_{F_A}{\frac 2{\lambda_A}}u\pi\langle y,\nu_A\rangle \,d\sigma_0-\int_{2P_+}u_{0,ij}^{ij}u\pi \,dy-2\int_{2P_+}u_{0,j}^{ij}\pi_{,i} u \,dy\notag\\
& \ \ \ -\int_{2P_+}4\langle\rho,\nabla\pi\rangle udy+\int_{2P_+}Q\pi\,dy.
\end{align}

On the other hand, by  (\ref{+5119}) and Lemma \ref{norm-convex}, we have
\begin{align}\label{43}
0\leq\int_{2P_+}4\langle\rho,\nabla{\pi}\rangle u\,dy&=\mathcal L(u)+\bar S\int_{2P_+}u\pi\,dy-\sum_A{\frac2{\lambda_A}}\int_{F_A}u\langle y,\nu_A\rangle\pi\,d\sigma_0\notag\\
&\leq \mathcal L(u)+C\int_{\partial(2P_+)}u\langle y,\nu\rangle\pi\,d\sigma_0.
\end{align}
 Moreover,
\begin{align}\label{u-0-ij}
\int_{2P_+}\left|u^{ij}_{0,ij}\right|u\pi \,dy\leq C_1\int_{2P_+}u\pi \,dy,
\end{align}
\begin{align}\label{u-0-ij-2}\int_{2P_+}\left|u^{ij}_{0,j}\pi_{,i}\right| u \,dy
&\leq C_2\int_{2P_+}\langle\rho,\nabla\pi\rangle u \,dy\notag\\
&\leq C_2\mathcal L(u)+C_2'\int_{\partial(2P_+)}u\langle y,\nu\rangle\pi\,d\sigma_0,
\end{align}
  since $u^{ij}_{0,ij}$, $u^{ij}_{0,j}$ are smooth up to the boundary, where $C_1, C_2, C_2'>0$ are constants independent of $u$.
  Hence, substituting (\ref{43}),  (\ref{u-0-ij}) and (\ref{u-0-ij-2}) into (\ref{49}), we obtain (\ref{+4N'}).
\end{proof}

\subsection{Estimate of $Q$}
Since $Q$ is singular and $\pi$ vanishes along each $W_\alpha$, we shall give an explicit estimate for the singular order of $Q$.
In the following, we will divide $2P$ into two parts $2P=2P'\cup U$, where $U$ is a union of small neighborhoods of faces of codim$\geq2$ which are contained in $\cup_{\alpha\in\Phi_+} W_\alpha$, and  $2P'_+= 2P'\cap \bar{\mathfrak a}^*_+$, where $2P'$  is a $W$-invariant polytope whose boundary intersects the Weyl walls orthogonally.
By Proposition \ref{non-linear part}, to finish the proof of Proposition \ref{proper-red}, it suffices to prove

\begin{prop}\label{Q int} There are constants $C_I, C_{II}>0$ independent of $u$ such that
\begin{equation}\label{604}
\left|\int_{2P_+}Qu\pi\, dy\right|
\leq C_I\int_{2P_+}\langle\rho,{\nabla\pi}\rangle u\, dy+C_{II}\int_{{2P_+}}u\pi\, dy,~\forall u\in\hat{\mathcal C}_{\infty,W}.\nonumber
\end{equation}
\end{prop}

\subsubsection{Integral estimate on $2P'_+$}
It is easy to see  that  $Q\pi$  is uniformly bounded in  $2P_+$. Then
\begin{equation}\label{q-pi}
\int_{2P'_+} Qu  \pi \,  dy\leq C \int_{2P'_+}u \, dy,\ \ \ \forall u\in\hat {\mathcal C}_{\infty,W}.
\end{equation}
In this subsection, we further prove

\begin{lem}\label{52}
Suppose that $2P'(\subset2P)$ is a $W$-invariant polytope as above.
Then there exists a constant $C_{P'}$ independent of $u$ such that
\begin{equation}\label{53}
\int_{2P'_+}u\,dy\leq C_{P'} \int_{2P'_+}u\pi \,dy,\ \ \ \forall u\in\hat {\mathcal C}_{\infty,W}.
\end{equation}
\end{lem}

\begin{proof}
Set $(2{P'}_+)_{\epsilon}:=(\cap_{\alpha\in\Phi_+}\{y|\langle\alpha,y\rangle>\epsilon\})\cap\overline{2 P'}_+$ for $\epsilon>0$. Then
\begin{equation}\label{54}
\pi(y)\geq\epsilon^{n-r},~\forall y\in(2{P'}_+)_{\epsilon},
\end{equation}
since the number of elements of $\Phi_+$ is  $\frac{n-r}{2}$.
Consequently
\begin{equation}\label{54}
\int_{2P'_+}u\,dy\leq\epsilon^{r-n}\int_{(2P'_+)_\epsilon}u\pi\,dy+\int_{2P'_+\setminus(2P'_+)_\epsilon}u\,dy,
\end{equation}
It suffices to estimate the second term for some fixed $\epsilon$.
Let ${y_0}\in \overline{2 P'_+}$ be a point which  lies on the intersection of exactly $k$ Weyl walls. For example,  ${y_0}\in\tilde W_k:=\cap_{i=1}^kW_{\alpha_{(i)}}$  and $y_0$ is away from other walls. Without loss of generality, we may assume that $\alpha_{(1)},...,\alpha_{(k)}$ are simple roots in $\Phi_+$. Then $\tilde W_k$ is an $(r-k)$-dim linear subspace in  $\mathfrak a^*$. Take $I_{y_0}$ a cubic relative neighbourhood of ${y_0}$ in $\tilde W_k\cap \overline{2 P'_+}$.
Consider the affine $k$-dim plane
$$H_{y_0}:=\{{y_0}+\sum_{i=1}^k\tau_i\alpha_{(i)}|\tau_i\in\mathbb R\},$$
which is the unique $k$-plane passing through $y_0$ and orthogonal to all
$W_{\alpha_{(1)}},...,W_{\alpha_{(k)}}$. By our assumptions, we can take a small relative neighbourhood $U_{y_0}$ of ${y_0}$ in $\overline{2 P'_+}\cap H_{y_0}$, which is an $k$-dimensional  polytope, such that $\partial U_{y_0}\backslash \cup_{i=1}^kW_{\alpha_{(i)}}$ intersects $W_{\alpha_{(1)}},...,W_{\alpha_{(k)}}$ orthogonally and is away from other Weyl walls. Let ${\frac 1 2}U_{y_0}$ be the shrinking of $U_{y_0}$ with centre at ${y_0}$ at rate ${\frac 1 2}$. Take $U_{y_0}$ small enough, one can assume $\Sigma_{y_0}:=U_{y_0}\times I_{y_0}$ and $\Sigma_{{y_0}}^{O}:={\frac 1 2}U_{y_0}\times I_{y_0}$ are contained in $\overline{2 P'_+}$, whose closures are away from other Weyl walls (See Figure 1).

\begin{figure}[hbp]
\centering
(a)\includegraphics[height=1.5in]{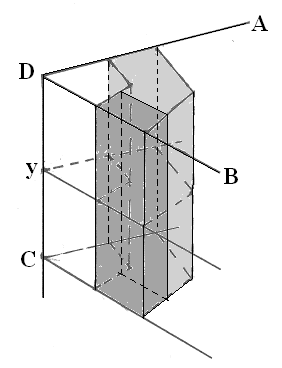}
(b)\includegraphics[height=1.5in]{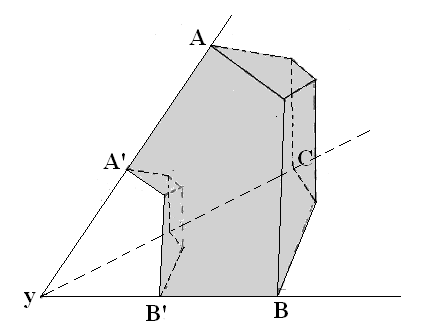}
\caption{The dark area is $\Sigma_{y_0}\backslash\Sigma_{y_0}^O$ in a  $3$-dimension $2P_+$.  In (a), ${y_0}$ lies on two walls $ADC$ and $BDC$,  $ADB$ is an outer face,  the line segment $\overline{DC}$ stands for $I_{y_0}$, and   the deeper dark
area  presents  a subpolytope $P_i$  of  case  (2) in (a).   In (b), ${y_0}$ lies on three walls.}
\end{figure}

Let $y=(y', y'')$ be any  point in $\Sigma_{y_0}$. Fix a $y''\in I_{y_0}$.  Since $u(y',y'')$ is a strictly convex function for $y'$, by the $W$-invariance of $u$, it must attains its minima at $\tilde y_0=( y_0',  y'')$, where ${y_0}'$ is the coordinate component  of ${y_0}$ in  $U_{y_0}$. By the convexity of $u$, we have
\begin{align}\label{55}
\int_{\Sigma_{y_0}^O}u\,dy&=\int_{I_{y_0}}\left(\int_0^{\frac 1 2}\int_{\partial U_{y_0}\backslash \tilde W_k}u(ty',y'')\langle y',\nu\rangle\,d\sigma_{\partial U_{y_0}}(y')\wedge dt\right )\wedge dy''\notag\\
&\leq \int_{I_{y_0}}\left(\int_{\frac 1 2}^1\int_{\partial U_{y_0}\backslash \tilde W_k}u(ty',y'')\langle y',\nu\rangle \,d\sigma_{\partial U_{y_0}}( y')\wedge dt\right)\wedge dy''\notag\\
&=\int_{\Sigma_{y_0}\backslash \Sigma_{y_0}^O}u\,dy.
\end{align}

We divide ${\Sigma_{y_0}\backslash \Sigma_{y_0}^O}$ into finitely many subpolytopes $P_1,...,P_m$ in two types:
\begin{itemize}
\item[(1)] $P_i$ is contained in some $(2P'_+)_{\epsilon}$;
\item[(2)] $\bar{P_i}$ intersects at most $(k-1)$ Weyl walls and its outer faces are orthogonal to these walls.
\end{itemize}
For $P_i$ of type  $(1)$, by (\ref{54}), we have
\begin{equation}\label{+54}
\int_{P_i}u\, dy\leq  \epsilon^{-(n-r)}\int_{P_i}u\pi\,dy,~\forall u\geq0.
\end{equation}
For $P_i$ of type $(2)$, we regard $P_i$ as $\Sigma_{y^{(1)}}$ for some $y^{(1)}$ which lies on at most $(k-1)$ Weyl walls. Then according to the above argument, there is a subset $\Sigma_{y^{(1)}}^O$ of $\Sigma_{y^{(1)}}$ such that as in (\ref{55}),
$$\int_{\Sigma_{y^{(1)}}^O}u \,dy\le \int_{\Sigma_{y^{(1)}}\setminus \Sigma_{y^{(1)}}^O}u\, dy.$$
Moreover, we have finitely many  subpolytopes $\{P_j^{(1)}\}_j$, where $P_j^{(1)}$ is either contained in some $(2P'_+)_{\epsilon_1}$ for some $\epsilon_1>0$,  or intersects at most $(k-2)$ Weyl walls  such that
$\Sigma_{y^{(1)}}\setminus \Sigma_{y^{(1)}}^O=\cup_j P_j^{(1)}$.

Thus we can iterate the above progress for finite times until each $P_j^{(k)}$ in  $\Sigma_{y^{(k)}}\setminus \Sigma_{y^{(k)}}^{O}$ is  of type (1) for some $\epsilon_k>0$ while $P_j^{(k-1)}$ is of type (2). Hence by the relations (\ref{55}) and (\ref{+54}), we can find a small number $\delta_0>0$ such that
\begin{equation}\label{near-wall}\int_{\Sigma_{y_0}}u \,dy\le C \delta_0^{-(n-r)} \int_{2P'_+}u\pi \, dy.\nonumber
\end{equation}
Since $\partial (2P_+)\cap (\cup_{\alpha\in\Phi_+} W_\alpha)$ is compact, we can cover it by finitely many $\{\Sigma_{y_p}\}$. Choose $\epsilon_0>0$ such that $2P'_+\setminus(2P'_+)_{\epsilon_0}\subset\cup_p\Sigma_{y_p}$.
Then (\ref{53})  follows from (\ref{54}).
\end{proof}

\begin{rem} If  $M$ is a toroidal compactification of $G$ \cite{Del3}, we can take $P'=P$ and then Proposition \ref{proper-red}   follows from Lemma \ref{52} directly.  Lemma \ref{52} will  be also  used in Section 6 (cf. Lemma \ref{extend u}).
\end{rem}

\subsubsection{Asymptotic estimate of $Q$ near $W_\alpha$ }

In general,  a Weyl wall $W_\alpha$   could not  intersect  a  $(r-1)$-dimensional  face $F_{\tilde A}$  of  $2P$ orthogonally.
In this case,  if let  $s_\alpha\in W_\alpha $ be the reflection with respect to $W_\alpha$, then by the $W$-invariance of $2P$, $F_{\tilde A,\alpha}:=s_\alpha(F_{\tilde A})$ is again a face of $2P$.
For simplicity, we  denote  $\mathfrak F_\alpha$ 
$$\mathfrak F_\alpha=\left\{F_{\tilde A}\subset\{y|~\langle \alpha, y\rangle\geq0 \}|~ F_{\tilde A} \neq F_{\tilde A,\alpha}\right\}.$$
We note that  $F_{\tilde A}$ may not intersect $W_\alpha$.

In order to make the computation of the quantity $Q$ more explicitly,   associated  to  each $W_\alpha$,  we relabel the $(r-1)$-dimensional faces of $2P$ as follows:

\begin{enumerate}
\item
Faces $ F_{a}\in \mathfrak F_\alpha$.  We denote them by
    $F_{a}=\{y\in\partial(2P)|~l_{a}(y)=0\},\,a=1,..,d_1$. By the convexity of $2P$, we have $\alpha(u_{a})>0$. Since $\alpha\in\mathfrak M$ and $u_a\in\mathfrak N$, $\alpha(u_a)\in\mathbb{Z}_{>0}$.

\item
Faces $ F_{a,\alpha}$ with $ F_{a}\in\mathfrak F_\alpha$.  We denote them by $F_{a,\alpha}=\{y\in\partial(2P)|~l_{a,\alpha}(y)=0\}$, where  $l_{a,\alpha}(y)$ satisfies
\begin{align}\label{symmetry}  l_{a,\alpha}(y)=l_{a}(y)+{\frac{2\alpha(u_{a})}{|\alpha|^2}}\langle\alpha,y\rangle.
\end{align}

\item
Faces $ F_{b}$ which are orthogonal to $W_\alpha$. By the convexity of $2P$, $F_b\cap W_\alpha\neq\emptyset$. We denote them by $F_b=\{y\in\partial(2P)|~l_{b}(y)=0\},\,b=1,..,d_2$.
Since $\alpha(u_b)=0$, $F_b$ is invariant under $s_\alpha$.
\end{enumerate}

Under the above notations,
we rewrite Guillemin's function $u_0$ in (\ref{Guillemin func}) as
$$u_0={\frac 1 2}\sum_{a}\left(l_{a}(y)\log l_{a}(y)+l_{a,\alpha}(y)\log l_{a,\alpha}(y)\right)+{\frac 1 2}\sum_bl_b(y)\log l_b(y).$$
Thus
\begin{equation}\label{D of Guillemin func}
\langle\alpha,\nabla u_0\rangle={\frac 1 2}\sum_{a}\alpha(u_a)\log\left(1+{\frac{2\alpha(u_a)\alpha(y)}{|\alpha|^2l_{a}(y)}}\right)
\end{equation}
and
\begin{equation}\label{D2 of Guillemin func}
{u_0}_{,ij}={\frac 1 2}\sum_{a}\left({\frac{u_a^iu_{a}^j}{l_{a}(y)}}+{\frac{u_{a,\alpha}^iu_{a,\alpha}^j}{l_{a,\alpha}(y)}}\right)
+{\frac 1 2}\sum_b\left({\frac{u_b^iu_b^j}{l_b(y)}}\right).
\end{equation}

\begin{lem}\label{u^-1}
Let  $y_0\in  W_\alpha$.  Then
\begin{equation}\label{+u1}
u_0^{ij}\alpha_i\alpha_j=\left\{
\begin{aligned}
&\left(\sum_{a}{\frac{(\alpha(u_{a}))^2}{l_{a}(y)}}\right)^{-1}|\alpha|^4+O(\alpha(y)), ~\text{if } {\frac{\alpha(y)}{l_{a}(y)}\to0}, ~\forall ~ a,\\[4pt]
&O(\alpha(y)),~\text{otherwise}.
\end{aligned}\nonumber
\right.
\end{equation}
\end{lem}

\begin{proof}
Since $l_{a,\alpha}(y)>l_{a}(y)>0$, we have
\begin{align}
0&<M_1:={\frac12}\left[\sum_{a}{\frac{1}{l_{a,\alpha}(y)}}\left(u_{a}^iu_{a}^j+{u_{a,\alpha}^iu_{a,\alpha}^j}\right)
+\sum_b{\frac{u_b^iu_b^j}{l_b(y)}}\right]\notag\\
&\leq (u_{0,ij})\leq{\frac12}\left[\sum_{a}{\frac{1}{l_{a}(y)}}\left(u_{a}^iu_{a}^j+{u_{a,\alpha}^iu_{a,\alpha}^j}\right)
+\sum_b{\frac{u_b^iu_b^j}{l_b(y)}}\right]=:M_2.\notag
\end{align}
It is easy to see
\begin{align}\label{eigenvalue}M_1\alpha=\left(\sum_{a}{\frac{(\alpha(u_{a}))^2 }{|\alpha|^2l_{a,\alpha}(y)}}\right)\alpha,\ \ \  M_2\alpha=\left(\sum_{a}{\frac{(\alpha(u_{a}))^2}{|\alpha|^2l_{a}(y)}}\right)\alpha.
\end{align}
Thus
\begin{align}
&\left(\sum_{a}{\frac{(\alpha(u_{a}))^2}{l_{a}(y)}}\right)^{-1}={\frac{\alpha^T M_2^{-1}\alpha}{|\alpha|^4}}\notag\\
&\leq {\frac{u_0^{ij}\alpha_i\alpha_j}{|\alpha|^4}}\leq {\frac{\alpha^T M_1^{-1}\alpha}{|\alpha|^4}}=
\left(\sum_{a}{\frac{(\alpha(u_{a}))^2}{l_{a}(y)+2{\frac{\alpha(u_{a})}{|\alpha|^2}}\alpha(y)}}\right)^{-1}.\nonumber
\end{align}
Moreover, if ${\frac{\alpha(y)}{l_{a}(y)}}\to0$ for all $a$,  by Lemma \ref{appendix} in Appendix,
\begin{align}\label{large-matrix} 0&\leq
\left(\sum_{a}{\frac{(\alpha(u_{a}))^2}{l_{a}(y)
+2{\frac{\alpha(u_{a})}{|\alpha|^2}}\alpha(y)}}\right)^{-1}
-\left(\sum_{a}{\frac{(\alpha(u_{a}))^2}{l_{a}(y)}}\right)^{-1}\notag\\[4pt]
&=O(\alpha(y)).
\end{align}
This implies
$${\frac{u_0^{ij}\alpha_i\alpha_j}{|\alpha|^4}}\leq \left(\sum_{a}{\frac{(\alpha(u_{a}))^2}{l_{a}(y)}}\right)^{-1}+O(\alpha(y)).$$
The first case is proved.

In the second case,   there exists an  $F_{a_0}\in\mathfrak F_\alpha$ such that
$\alpha(u_{a_0})\not=0$ and ${\frac{\alpha(y)}{l_{a_0}(y)}}\geq\epsilon_0$ for some $\epsilon_0>0$. Then
\begin{align}\label{large-matrix-2}
\left(\sum_{a}{\frac{(\alpha(u_{ a}))^2}{l_{a}(y)}}\right)^{-1}=O(l_{a_0}(y))\end{align}
and
\begin{align}\label{large-matrix-3}\left(\sum_{a}{\frac{(\alpha(u_{a}))^2}{l_{a}(y)+2{\frac{\alpha(u_{ a})}{|\alpha|^2}}\alpha(y)}}\right)^{-1}=O(\alpha(y)).
\end{align}
Thus
$${\frac{u_0^{ij}\alpha_i\alpha_j}{|\alpha|^4}}\le  O(l_{a_0}(y))+O(\alpha(y)) = O(\alpha(y)).$$
 The lemma is proved.
\end{proof}

\begin{lem}\label{M_2}
Let  $y_0\in  W_\alpha $. Suppose that $y_0$ also lies on  another Weyl wall $W_\beta$.
Then as $y\to y_0$,  it holds
\begin{align}\label{m2-u-0}
\alpha^T\left((u^{ij}_0)-M_2^{-1}\right)\alpha=& O(\alpha(y)),
\end{align}\label{m2-beta}
\begin{align}\label{m2-u-0-beta}\beta^T\left((u^{ij}_0)-M_2^{-1}\right)\beta= O(\alpha(y)+\beta(y)).
\end{align}
\end{lem}

\begin{proof}
 (\ref{m2-u-0}) follows from  the estimate in Lemma \ref{u^-1} immediately.  It remains to prove (\ref{m2-u-0-beta}).
Let $S_{\alpha,\beta}\subset W$ be the group generated by the reflections $s_\alpha$ and  $s_\beta$. We want to relabel faces of $2P$ according to this $S_{\alpha,\beta}$-action. In each orbit $\{ S_{\alpha,\beta}F_{\tilde A}\}$,
where $\{F_{\tilde A}\}$ is a $(r-1)$-dimensional face,  we take a face $F_c$ such that $\alpha(u_c),\,\beta(u_c)\geq0$. Let
$$ \{ S_{\alpha,\beta}F_{c}\}=\{F_{c,s_1},...,F_{c,s_{p(c)}}\},$$ where $F_{c,s}=\{y\in \partial(2P)|~ l_{c,s}(y)=\lambda_{c,s}-u_{c,s}^i y_i=0\}.$
Set
$$\hat M_1={\frac12}\sum_c\left({\frac{\sum_{q=1}^{p(c)}u_{c,s_q}^iu_{c,s_q}^j}{\max_{s}\left\{\left(l_{c,s}\right)(y)\right\}}}\right),~\hat M_2={\frac12}\sum_c\left({\frac{\sum_{q=1}^{p(c)}u_{c,s_q}^iu_{c,s_q}^j}{\min_{s}\left\{\left(l_{c,s}\right)(y)\right\}}}\right).$$
Then we  rewrite $u_0$  as
$$u_0={\frac12}\sum_c\sum_{q=1}^{p(c)}{\frac{u_{c,s_q}^iu_{c,s_q}^j}{l_{c,s_q}(y)}}.$$
Thus it is easy to see
$$0<\hat M_1\leq(u_{0,ij})\leq M_2\leq\hat M_2.$$

Note that   $ \{S_{\alpha,\beta}F_{c}\}$ is invariant  under the reflection associated to $W_\beta$.  Then as in the proof of Lemma \ref{u^-1}, we relabel $ \{ S_{\alpha,\beta}F_{c}\}$: faces $ F_{a}\in \mathfrak F_\beta$; faces $ F_{a,\beta}$ with $ F_{a}\in \mathfrak F_\beta$; and Faces $ F_{b}$ %with  $ F_{b}\cap W_\beta\neq\emptyset$,
which is orthogonal to $W_\beta$. Thus similar to (\ref{eigenvalue}), we get
$$\sum_{q=1}^{p(c)}u_{c,s_q}^iu_{c,s_q}^j\beta_j=\lambda_c\beta_i,$$
where  $\lambda_c\geq0$  is a constant with  at least one $\lambda_c>0$ since $\hat M_1>0$.
As a consequence,
$$\hat M_1\beta={\frac12}\sum_c\frac{\lambda_c}{\max_{s}\left\{l_{c,s}(y)\right\}}\beta $$
and
$$\hat M_2\beta={\frac12}\sum_c\frac{\lambda_c}{\min_{s}\left\{l_{c,s}(y)\right\}}\beta.$$
This means that  $\beta$ is an eigenvector of both $\hat M_1$ and $\hat M_2$.

On the other hand, there are  constants $\mu_1^c,\,\mu_2^c$ such that
$$l_{c,s}(y)=l_c(y)+\mu_1^c \alpha(y)+\mu_2^c \beta(y).$$
In particular,
$$ \max_{s}\{l_{c,s}(y)\}=\min_{s}\{l_{c,s}(y)\}
 +\tilde\mu_1^c \alpha(y)+  \tilde\mu_2^c\beta(y).$$
Then as in the estimate (\ref{large-matrix}) (also see (\ref{large-matrix-2}),  (\ref{large-matrix-3})),  we get
 $$\beta^T\left(\hat M_1^{-1}-\hat M_2^{-1}\right)\beta=O(\alpha(y)+\beta(y)).$$
 It follows
$$\beta^T\left((u^{ij}_0)-M_2^{-1}\right)\beta=O(\alpha(y)+\beta(y)).$$
The lemma is proved.

\end{proof}

\begin{rem}\label{remark-w-beta} Since   (\ref{D2 of Guillemin func}) can be rewritten as
$${u_0}_{,ij}={\frac 1 2}\sum_{a}\left({\frac{u_{a}^iu_{a}^j}{l_{a}(y)}}+{\frac{u_{a,\beta}^iu_{a,\beta}^j}{l_{a,\beta}(y)}}\right)+{\frac 1 2}\sum_b\left({\frac{u_b^iu_b^j}{l_b(y)}}\right)$$
for any Weyl wall $W_\beta$,  as in the proof of Lemma \ref{u^-1} for the second case, one can prove:
if  $y_0\notin  W_\beta$,  then
\begin{align}\label{y-not}
{u}^{ij}_0\beta_i\beta_j=O(1),~{\rm as}~ y\to y_0.
\end{align}
Similar to (\ref{m2-beta}),
\begin{align}\label{not-in-2}\beta^T\left((u^{ij}_0)-M_2^{-1}\right)\beta= O(1),~{\rm if}~y_0\in W_\alpha\setminus  W_\beta.
\end{align}

\end{rem}

 From (\ref{59}),   a  direct computation shows
\begin{align}\label{Q ex}
Q&=\sum_{\alpha\in\Phi_+}\left[4{\frac{|\alpha|^2\coth \langle\alpha,\nabla u_0\rangle}{\alpha(y)}}-2{\frac{u_{0,ij}\alpha_i\alpha_j}{\sinh^{2}\langle\alpha,\nabla u_0\rangle}}-2u_0^{ij}{\frac{\alpha_i\alpha_j}{(\alpha(y))^2}}\right]\nonumber\\
&+2\sum_{\alpha\neq \beta\in\Phi_+}\left[{\coth{\langle\alpha,\nabla u_0\rangle}}{\frac{\langle \alpha,\beta\rangle}{\beta(y)}}+{\coth{\langle\beta,\nabla u_0\rangle}}{\frac{\langle \alpha,\beta\rangle}{\alpha(y)}}-2u_0^{ij}{\frac{\alpha_i\beta_j}{\alpha(y)\beta(y)}}\right].
\end{align}
For simplicity, we denote each term  in these  two sums  by  $ I_\alpha(y)$  and $I_{\alpha,\beta}(y)$, respectively.
We need to estimate  them in the following key lemma.

\begin{lem}\label{521}Let  $y_0\in W_{\alpha}$.   Then there exist $C_{\alpha,y_0}, C_{\alpha,\beta,y_0}>0$, such that
\begin{equation}\label{Ialpha}
|I_\alpha(y)|\leq {\frac{C_{\alpha,y_0}}{\alpha(y)}}
\end{equation}
and
\begin{equation}\label{Ialphabeta}
|I_{\alpha,\beta}(y)|\leq C_{\alpha,\beta,y_0}\left({\frac{1}{\alpha(y)}}+{\frac{1}{\beta(y)}}\right)
\end{equation}
as  $y\to y_0$.
\end{lem}
\begin{proof}

We  consider  the following three cases as $y\to y_0$:
\begin{itemize}
\item [(i)] ${\frac{\alpha(y)}{l_{a}(y)}}\le  \epsilon_0<<1$, $\forall  F_a \in \mathfrak F_\alpha$.
\item [(ii)]There is an $F_{a_0}\in \mathfrak F_\alpha$  such that  $0<{\frac{1}{\tau}}<{\frac{\alpha(y)}{l_{a_0}(y)}}<\tau$ for some $0<\tau<+\infty$, and ${\frac{\alpha(y)}{l_{a}(y)}}\le \tau,~\forall  F_a$.
\item [(iii)]There is  an $F_{a_0}\in \mathfrak F_\alpha$ such that  ${\frac{\alpha(y)}{l_{a_0}(y)}}\ge N_0>>1$.
\end{itemize}

\textbf{Case (i)}.
In this case, $\langle\alpha,\nabla u_a\rangle\to 0$. By (\ref{D of Guillemin func}),  it is easy to check
$$\langle\alpha,\nabla u_0\rangle=\sum_{a}{\frac{(\alpha(u_{a}))^2\alpha(y)}{|\alpha|^2l_{a}(y)}}+O\left(\sum_{a}\left({\frac{(\alpha(u_{a}))^2\alpha(y)}{|\alpha|^2l_{a}(y)}}\right)^2\right).$$
Then
\begin{eqnarray*}
\coth\langle\alpha,\nabla u_0\rangle&=&{\frac{|\alpha|^2}{\alpha(y)}}\left(\sum_{a}{\frac{(\alpha(u_{a}))^2}{l_{a}(y)}}\right)^{-1}+O(1).\end{eqnarray*}
On the other hand, by (\ref{D2 of Guillemin func}), one can show
$$u_{0,ij}\alpha_i\alpha_j= \left(\sum_{a}{\frac{(\alpha(u_{a}))^2}{l_{a}(y)}}\right) + O\left(\sum_{a}{\frac{\alpha(y)}{l_{a}(y)l_{a,\alpha}(y)}}\right).$$
Then
\begin{eqnarray*}
 {\frac{ u_{0,ij}\alpha_i\alpha_j}{\sinh^{2}\langle\alpha,\nabla u_0\rangle}}
={\frac{|\alpha|^4}{(\alpha(y))^2}}\left(\sum_{a}{\frac{(\alpha(u_{a}))^2}{l_{a}(y)}}\right)^{-1}+O\left({\frac{1}{\alpha(y)}}\right).
\end{eqnarray*}
Thus
\begin{align}
&4{\frac{|\alpha|^2\coth \langle\alpha,\nabla u_0\rangle}{\alpha(y)}}-2{\frac{ u_{0,ij}\alpha_i\alpha_j}{\sinh^{2}\langle\alpha,\nabla u_0\rangle}}\notag\\
&=2{\frac{|\alpha|^4}{(\alpha(y))^2}}\left(\sum_{a}{\frac{(\alpha(u_{a}))^2}{l_{a}(y)}}\right)^{-1}+O\left({\frac{1}{\alpha(y)}}\right).\notag
\end{align}
Hence, by the above relation and  Lemma \ref{u^-1},  we  see that there exists a constant $C'>0$ such that
 $|I_\alpha(y)|\leq {\frac{C'}{\alpha(y)}}$.

\textbf{Case (ii)}.
In this case,  it is easy to see
 $$0<\coth\langle\alpha,\nabla u_0\rangle,\, \frac1{\sinh^2\langle\alpha,\nabla u_0\rangle}\leq C'_\tau, ~{\rm and}~
 u_{0,ij}\alpha_i\alpha_j= O\left(\frac{1}{\alpha(y)}\right).$$
Then by Lemma \ref{u^-1}, we have
$|I_\alpha(y)|\leq {\frac{C_{\tau}}{\alpha(y)}}$,
where  the  constant $C_{\tau}<\infty$ depends only on $\tau$.

\textbf{Case (iii)}.
In this case,   we may assume
 $$\frac{\alpha(y)}{l_{a}(y)}\le \frac{\alpha(y)}{l_{a_0}(y)}\ge N_0,~\forall a=1,..., d_1.$$
Then  by (\ref{D of Guillemin func}), we have
\begin{align}\label{gradient-order} \langle\alpha,\nabla u_0\rangle \ge {\frac 1 2}\alpha(u_{a_0})
\log\left(1+{\frac{2\alpha(u_{a_0})\alpha(y)}{|\alpha|^2l_{a_0}(y)}}\right).
\end{align}
It follows
\begin{align}\label{O-nabla}\coth \langle\alpha,\nabla u_0\rangle
=O(1)\end{align}
and
$$\sinh^{2}\langle\alpha,\nabla u_0\rangle\ge   \left(1+{\frac{2\alpha(u_{a_0})\alpha(y)}{|\alpha|^2l_{a_0}(y)}}\right)^{\alpha(u_{a_0}) }.$$
On the other hand,  by (\ref{D2 of Guillemin func}),  it is easy to see
$$ u_{0,ij}\alpha_i\alpha_j\le  2d_1  \left({\frac{(\alpha(u_{a_0})^2)}{l_{a_0}(y)}}\right) +C.$$
Thus
\begin{equation}\label{second-derivative}
{\frac{u_{0,ij}\alpha_i\alpha_j}{\sinh^{2}\langle\alpha,\nabla u_0\rangle}}
 =O\left({\frac 1{\alpha(y)}\left({\frac{l_{a_0}(y)}{\alpha(y)}}\right)^{\alpha(u_{a_0})-1}}\right)
\leq o\left({\frac 1{\alpha(y)}}\right).
\end{equation}
Here we used the fact that $\alpha(u_{a})\in\mathbb{Z}_{>0}$, hence $\geq1$.
Hence, combining (\ref{O-nabla}) and (\ref{second-derivative}) together with Lemma \ref{u^-1}, we get
$I_\alpha(y)=O\left(\frac1{\alpha(y)}\right)$. The proof of (\ref{Ialpha}) is completed.

Next,  we prove (\ref{Ialphabeta}). We may
assume $y_0\in W_\alpha\cap W_\beta$,  otherwise,   (\ref{Ialphabeta}) can be more easy to obtained (cf. Remark \ref{remark-w-beta}).
We note that $(u_0^{ij})-M_2^{-1}\geq0$ and $\alpha$ is an eigenvector of  $M_2$.  Then by  the above discussion  for $\coth\langle\alpha,\nabla u_0\rangle$  in cases (i)-(iii), we have
\begin{eqnarray*}
\coth\langle\alpha,\nabla u_0\rangle\cdot\langle\alpha,\beta\rangle-{\frac{\langle M_2^{-1}\alpha,\beta\rangle}{\alpha(y)}}=O(1).
\end{eqnarray*}
On the other hand, by   Lemma \ref{M_2},
\begin{eqnarray*}
\left|{\frac{\alpha^T((u_0^{ij})-M_2^{-1})\beta}{\alpha(y)}}\right|\leq{\frac{O(\alpha(y)+\beta(y))}{\alpha(y)}}.
\end{eqnarray*}
Thus
\begin{equation}
\begin{aligned}
&\coth{\langle\alpha,\nabla u_0\rangle}{\frac{\langle\alpha,\beta\rangle}{\beta(y)}}-u_0^{ij}{\frac{\alpha_i\beta_j}{\alpha(y)\beta(y)}}
=O\left({\frac{1}{\alpha(y)}}+{\frac{1}{\beta(y)}}\right),~y\to y_0.\nonumber
\end{aligned}
\end{equation}
Similarly, we have
\begin{equation}
\begin{aligned}
&\coth{\langle\beta,\nabla u_0\rangle}{\frac{\langle\alpha,\beta\rangle}{\alpha(y)}}-u_0^{ij}{\frac{\alpha_i\beta_j}{\alpha(y)\beta(y)}}
=O\left({\frac{1}{\alpha(y)}}+{\frac{1}{\beta(y)}}\right),~y\to y_0.\nonumber
\end{aligned}
\end{equation}
Combining these two relations, we see that (\ref{Ialphabeta}) is true.
\end{proof}

By Lemma \ref{52}  and  Lemma \ref{521}, we  begin to  prove  Proposition \ref{Q int}.

\begin{proof}[Proof of Proposition \ref{Q int}]
Set a  compact  subset of $\partial (2P_+)$  by
$$\Omega^*=\cup_{\alpha\in\Phi_+}\{y\in W_\alpha\cap F_{\tilde A}|~ W_\alpha ~{\rm intersects}~ F_{\tilde A} ~{\rm not~orthogonally}\}.$$
%Let $\Omega^*_+:=\Omega^*\cap \bar{\mathfrak a}_+$. 
Since $\emptyset\neq F_{\tilde A}\cap F_{\tilde A,\alpha}\subset W_\alpha$ if $ F_{\tilde A}\cap W_\alpha\neq\emptyset$, each point in $\Omega^*$ lies on  a face of codimension greater than 2.
We claim:
for any $y_0\in\Omega^*\cap\overline{(2P_+)}$,   there is a neighbourhood $V_{y_0}$ and a constant $C_{y_0}>0$ such that
\begin{equation}\label{522}
|Q|\leq C_{y_0}\left\langle\rho,{\frac{\nabla\pi}{\pi}}\right\rangle,\ \forall y\in V_{y_0}\cap(2P_+).
\end{equation}

By (\ref{dpi}), we see there exists a uniform $C$ such that
$$\left\langle {\frac{\nabla\pi}{2\pi}},\rho\right\rangle=\sum_{\alpha\in\Phi_+}{\frac{\langle\alpha,\rho\rangle}{\alpha(y)}}
\ge C\sum_{\alpha\in\Phi_+}{\frac{ 1}{\alpha(y)}}, $$
since  $\langle\alpha,\rho\rangle>0$ for each $\alpha$.  Thus  by Lemma \ref{521},  to prove the claim,  it suffices to estimate $I_\beta(y)$
with $y_0 \notin W_\beta$ and $I_{\beta,\gamma}(y)$ with $y_0 \notin W_\beta\cup W_\gamma$.  The later can be easily settled. In fact, $I_{\beta,\gamma}(y)$ is bounded near $y_0$.
  For $I_\beta(y)$, we observe that  $\langle\beta, \nabla u_0\rangle\ge C_0>0$ for any $y\in V_{y_0}$.  As in the proof of   Lemma \ref{521} for \textbf{Case (iii)},  we have
$$\coth \langle\beta,\nabla u_0\rangle
=O(1)$$
and
\begin{align}{\frac{u_{0,ij}\beta_i\beta_j}{\sinh^{2}\langle\beta,\nabla u_0\rangle}}=O\left({\frac 1{\beta(y)}\left({\frac{l_{a_0}(y)}{\beta(y)}}\right)^{\beta(u_{a_0})-1}}\right)\leq O(1).\nonumber
\end{align}
Hence,  together with (\ref{y-not}) in Remark \ref{remark-w-beta}, we get
\begin{align}\label{i-beta}
|I_\beta(y)|\le C, ~\forall y\in V_{y_0}\cap(2P_+).
\end{align}
The claim is proved.

By the above claim, we can pick a small neighbourhood $U$ of $\bar \Omega^*$ in $2P$ and a constant $C_U<+\infty$ independent of $u$, such that
\begin{equation}\label{603}
|Q|\leq C_U\left\langle\rho,{\frac{\nabla\pi}{\pi}}\right\rangle,\,\ \forall y\in U\cap(2P_+).
\end{equation}
Furthermore, we can take a $W$-invariant polytope $P'$ whose boundary intersects the Weyl walls orthogonally and $2P\backslash2P'\subset U$. This can be done as follows: for any $F_{\tilde A}\cap W_{\alpha_{(1)}}\cap...\cap W_{\alpha_{(k)}}\subset \Omega^*$, we chop off a sufficiently small corner of $2P$ with
$F_{\tilde A}\cap W_{\alpha_{(1)}}\cap...\cap W_{\alpha_{(k)}}$ being the top and with the base lying on an $(r-1)$-plane which is parallel to $F_{\tilde A}\cap W_{\alpha_{(1)}}\cap...\cap W_{\alpha_{(k)}}$ and orthogonal to $W_{\alpha_{(1)}},...,W_{\alpha_{(k)}}$. By suitable choice of the chopping off,  $2P'$ is $W$-invariant (See Figure 2).

\begin{figure}
\centering
(a)\includegraphics[height=1.5in]{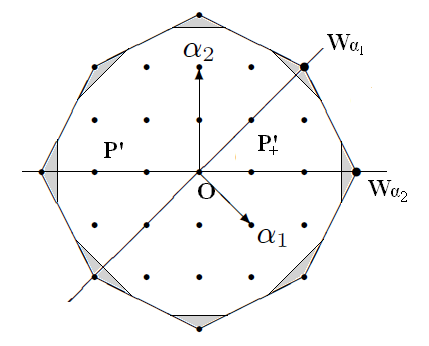}
(b)\includegraphics[height=1.5in]{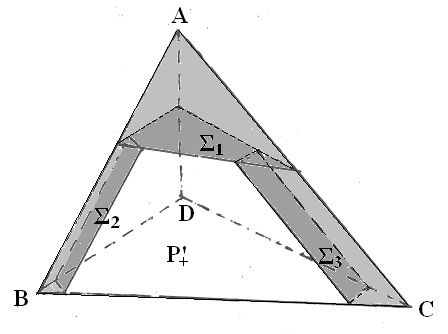}
\caption{The dark areas stand for $U$. In (a), $P$ is of dimension $2$ and we present out  the whole $P$ and $U$.   In (b), $P$ is of  dimension $3$.  We only present out  $P_+$.  $ADB,~ADC$ are two walls and $ABC,~BDC$ are outer faces. For simplicity, we assume  that $BDC$ is orthogonal to both walls, so that we need not to cut $P_+$ near $\overline{BD}$ and $\overline{CD}$.}
\end{figure}
Note that  $|Q\pi|$ is uniformly  bounded on $2P$.  Thus  by Lemma \ref{52} and (\ref{603}), we obtain
\begin{align}
\left|\int_{2P_+}Qu\pi\, dy\right|
&\leq\left(\int_U+\int_{{2P'_+}}\right)|Qu\pi|  dy\notag\\
&\leq C_U\int_{2P_+}\langle\rho,{\nabla\pi}\rangle u\, dy+\hat C\int_{{2P'_+}}u\pi dy.\notag
\end{align}
 Proposition  \ref{Q int} is proved.
\end{proof}

\begin{rem}\label{remark-Q}According to the proof in  Lemma \ref{521}, we  actually  prove that (\ref{522}) holds  for any $y_0\in \overline{2P_+}$. Then we can avoid to use Lemma \ref{52} and improve Proposition  \ref{Q int}  by the following  estimate
$$\left|\int_{2P_+}Qu\pi\, dy\right|
\leq C\int_{2P_+}\langle\rho,{\nabla\pi}\rangle u\, dy, ~\forall u\in\hat{\mathcal C}_{\infty,W}.$$

\end{rem}

\begin{proof}[Proof of Proposition \ref{proper-red}]
Let $f(t)=t-\log\sinh t,\,t>0$. Then
$$0>f(t)-f(\epsilon t)>\log \epsilon.$$
Regarding  $\mathcal N(\cdot)$ as $f(t)$,  and  then by Proposition \ref{non-linear part}, it follows
\begin{equation}\label{607}
\mathcal N(u)>\mathcal N(\epsilon u)-n\log \epsilon
\geq C_0-\epsilon\mathcal L_B(u)-n\log \epsilon,
\end{equation}
where
\begin{equation}\label{601}
\mathcal L_B(u)=-C_{\Lambda}\int_{\partial(2P_+)}u\langle y,\nu\rangle\pi\, d\sigma_0-C_L\mathcal L(u)+\int_{2P_+}Qu\pi\, dy.
\nonumber
\end{equation}
On the other hand, by Lemma \ref{norm-convex} and Propositions {\ref{5301}} and \ref{Q int}, for any $\delta\in(0,1)$, there exists uniform constants $C_1$, $C_2$, $C_3$ independent of $u$ such that
\begin{align}\label{605}
\mathcal L_B(u)-\mathcal L(u)&\leq C_1\int_{(2P_+)}u\pi \,dy+C_2\int_{\partial(2P_+)}u\langle y,\nu\rangle\pi \,d\sigma_0+C_3\mathcal L(u)\notag\\
&\leq\left(C_1\Lambda+C_2\right)\int_{\partial{(2P_+)}}u\langle y,\nu\rangle\pi \,d\sigma_0+C_3\mathcal L(u)\notag\\
&\leq\left(C_1\Lambda+C_2+\delta\Lambda\right)\int_{\partial{(2P_+)}}u\langle y,\nu\rangle\pi \,d\sigma_0+C_3\mathcal L(u)-\delta\int_{2P_+}u\pi \,dy\notag\\
&\leq\left(C_3+{\frac{C_1\Lambda+C_2+\delta\Lambda}{\lambda}}\right)\mathcal L(u)-\delta\int_{2P_+}u\pi \,dy.
\end{align}
Thus by choosing  $\epsilon=\left[1+\left(C_3+{\frac{C_1\Lambda+C_2+\delta\Lambda}{\lambda}}\right)\right]^{-1}$,  we get
$$\mathcal L_B(\epsilon u)<\mathcal L(u)-\epsilon\delta\int_{2P_+}u\pi\, dy.$$
By (\ref{607}),  we derive
$$\mathcal K(u)\geq \epsilon\delta\int_{2P_+}u\pi\, dy-C_{\delta},$$
where $C_{\delta}$ is  independent of $u$. (\ref{608}) is proved by replacing $\epsilon\delta$ with $\delta$.
\end{proof}

\subsection{Proof of Theorem \ref{LZZ1}}

Recall that the $J$-functional is given by
\begin{equation}\label{+3102}
\begin{aligned}
 J_{\omega_0}(\phi)&={\frac 1 {V_M}}\int_0^1\int_{M}\dot{\phi}_t(\omega_0^n-\omega_{\phi_t}^n)\wedge dt,\nonumber
\end{aligned}
\end{equation}
where $\phi\in\mathcal H_{K\times K}(\omega_0)$ and $\phi_t$ is a path in $\mathcal H_{K\times K}(\omega_0)$ joining $0$ and $\phi$.
The following definition can be found in  \cite{Tian, ZZ, DR}, etc.

\begin{defi}\label{properdef}
$\mu_{\omega_0}(\phi)$ is called proper modulo a subgroup $G_0$ of $\text{Aut}(M)$ in K\"ahler class $[\omega_0]$  if there is a continuous function $p(t)$ on $\mathbb R$ with the property $\displaystyle\lim_{t\to+\infty} p(t)=+\infty$, such that
\begin{equation}
\mu_{\omega_0}(\phi)\ge \inf_{\sigma\in G_0} p( J_{\omega_0}(\phi_{\sigma})),\nonumber
\end{equation}
where
$\phi_{\sigma}$ is defined by
$\omega_0+\sqrt{-1}\partial\bar{\partial}\phi_{\sigma}
=\sigma^*(\omega_0+\sqrt{-1}\partial\bar{\partial}\phi).$
\end{defi}

For our purpose, we focus on $\phi\in \mathcal H_{K\times K}(\omega_0)$ and $ G_0=Z(G)$. Let $u$ be the Legendre function of $\psi_0+\phi$.
Take a $v\in\eta_c(M)$ such that Re$(v)=-\nabla u(O)$. Let $\sigma_t^v$ be a one parameter group generated by Re$(v)$. Then $(\sigma_t^v)\in Z(G)$. It follows
$$(\sigma^v_1)^*\omega_\phi=\omega_0+ \sqrt{-1}\partial\bar\partial \tilde\phi$$
induces a $K\times K$-invariant K\"ahler potential $\tilde\phi$.
Thus  the Legendre function  $\tilde u$  of $\psi_0+\tilde \phi$ satisfies  $\nabla\tilde u(O)=0$. Moreover,  $\nabla(\psi_0+\tilde{\phi})(O)=0$. Since we may also normalize $\psi_0+\tilde\phi$ so that $(\psi_0+\tilde{\phi})(O)=0$, thus $\tilde u(O)=0$.
Moreover, $\mathcal K(\tilde u)=\mathcal K(u)$ since $\mathcal L(\tilde u)=\mathcal L(u)$ by Lemma \ref{5205} and the vanishing of Futaki invariant.

The following lemma is an analogue to \cite[Lemma 2.2]{ZZ}.
\begin{lem}\label{609}
There exists a uniform $C_J>0$ such that
\begin{equation}\label{bound psi}
\left | J_{\omega_0}(\tilde\phi)-{\frac 1 V}\int_{2P_+}\tilde u\pi\, dy\right|\leq C_J, \ \forall \phi\in \mathcal H_{K\times K}(\omega_0),\nonumber
\end{equation}
where $\tilde u\in\hat{\mathcal C}_{\infty,W}$ and $\psi_0+\tilde{\phi}$ is the Legendre function of $\tilde u$.
\end{lem}

\begin{proof}
In fact, Lemma \ref{609} comes from the following new version of $\mathcal J_{\omega_0}(\phi)$,
\begin{eqnarray*}
 J_{\omega_0}(\phi)
&=&{\frac{1}{V_M}}\int_M\phi\,\omega_0^n-{\frac{1}{V_M}}\int_0^1\int_M\dot{\phi}_t\,\omega_{\phi_t}^n\wedge dt\\
&=&{\frac{1}{V_M}}\int_M\phi\,\omega_0^n-{\frac{1}{V}}\int_0^1\int_{\mathfrak a_+}\dot{\phi_t}MA_{\mathbb R}(\psi_t)\prod_{\alpha\in\Phi_+}\langle\alpha,\nabla\psi_t\rangle^2\,dx\wedge dt\\
&=&{\frac{1}{V_M}}\int_M\phi\,\omega_0^n+{\frac{1}{V}}\int_{2P_+}(u-u_0)\pi\,dy.
\end{eqnarray*}
Then the lemma can be proved similarly as Lemma 2.2 in  \cite{ZZ}.
\end{proof}

\begin{proof}[Proof of Theorem \ref{LZZ1}]
For any $\phi\in \mathcal H_{K\times K}(\omega_0)$, there exists $\sigma\in Z(G)$ such that
$$\sigma^*\omega_\phi=\omega_0+ \sqrt{-1}\partial\bar\partial \tilde\phi$$
as above.
Applying Proposition \ref{proper-red}, we have
$$\mathcal K(\tilde u)\geq  \delta\int_{2P_+}\tilde u\pi dy-C_{\delta}.$$
Thus  by  Proposition \ref{5104} and Lemma \ref{609}, we get
$$\mu_{\omega_0}(\phi)=\mu_{\omega_0}(\phi)=\frac{1}{V}\mathcal K(\tilde u)  \geq \delta\cdot   J_{\omega_0}(\tilde\phi)-C_J-{\frac{C_{\delta}}{V}}.$$
The theorem is proved.
\end{proof}

\section{K\"ahler-Ricci solitons and the Modified K-energy}
In this section, we   verify the properness of modified K-energy on $(M, K_M^{-1})$  under an analogous condition of (\ref{bar}).
By Hodge theorem, for any $v\in \eta(M)$, there exists  a unique smooth complex-valued  function
$\theta_v(\omega_\phi)$ of $M$ such that
\begin{equation}\label{potential-vec}
i_v\omega_\phi=\sqrt{-1}\bar\partial\theta_v(\omega_\phi), \ \int_Me^{\theta_v(\omega_\phi)}\omega_\phi^n=\int_M\omega_\phi^n.\nonumber
\end{equation}
If $\phi\in\mathcal H_{K\times K}(\omega_0)$ and $v\in\eta_c(M)$, $\theta_v(\omega_\phi)$ is $K\times K$-invariant, so it can be written as
\begin{equation}\label{+5201}
\theta_{X}(\omega_{\phi})=c^i{\frac{\partial\psi}{\partial x^i}}+c,~\forall x\in\mathfrak a,\nonumber
\end{equation}
where $c^i$ and $c$ are constants with $c^i\alpha_i=0$ for any $\alpha\in\Phi_+$. Since the soliton vector field $X\in\eta_c(M)$ and Im$(X)\in\mathfrak k_r$, we have $c^i,\,c\in\mathbb R$. Furthermore, by the vanishing of the modified Futaki invariant \cite{TZ02}, they can be uniquely determined by the following linear equations,
\begin{eqnarray}
\int_{2P_+}(c^iy_i+c)\pi\, dy&=&0,\nonumber\\
\left\langle v,\int_{2P_+}ye^{c^iy_i+c}\pi\, dy\right\rangle&=&0,\ \ \ \forall v\in\mathfrak a^*_{t}.\label{coefficients of X 2}
\end{eqnarray}

The modified K-energy $\mu_{\omega_0}^X(\cdot)$ associated to $X$ is defined by
\begin{equation}\label{4101}
\mu_{\omega_0}^X(\phi)={\frac 1 {V_M}}\int_{M}\log{\left({\frac{\omega_{\phi}^n}{\omega_0^n}}e^{\phi-h_0}\right)}e^{\theta_X(\omega_{\phi})}\omega_{\phi}^n-{\frac 1 {V_M}}\int_0^1\int_{M}\dot{\phi_t}e^{\theta_X(\omega_{\phi_t})}\omega_{\phi_t}^n\wedge
dt,\nonumber
\end{equation}
where $\phi\in\mathcal H_X(\omega_0)$ and $\phi_t$ is a path in $\mathcal H_X(\omega_0)$ joining 0 and $\phi$ \cite{TZ02}.
The modified J-functional is defined by
\begin{equation}\label{AubinMJ}
 J_{\omega_0}^X(\phi)={\frac 1 {V_M}}\int_0^1\int_M\dot{\phi}_t\left(e^{\theta_X(\omega_0)}\omega_0^n-e^{\theta_X(\omega_{\phi_t})}\omega_{\phi_t}^n\right)\wedge dt.\nonumber
\end{equation}
The properness of $\mu_{\omega_0}^X(\cdot)$ can be defined analogous to Definition \ref{properdef} \cite{CTZ}.

The following is the main result of this section.

\begin{theo}\label{thm40}
Let $M$ be a Fano compactification of $G$ and $X$  the soliton vector field as above.
Let $$bar_X:={\frac{\int_{2P_+}y e^{\theta_X(y)}\pi\,dy}{\int_{2P_+} e^{\theta_X(y)}\pi\,dy}},$$
where $\theta_X(y)=c^iy_i+c$. Suppose that the corresponding polytope $2P_+$ satisfies
\begin{equation}\label{bar4}
bar_X\in4\rho+\Xi.
\end{equation}
Then $\mu_{\omega_0}^X(\cdot)$ is proper on $\mathcal H_{K\times K}(\omega_0)$ modulo  $Z(G)$.
\end{theo}

Since the properness of the modified K-energy implies the existence of K\"ahler-Ricci solitons \cite{TZ},  Theorem \ref{thm40} gives a proof for the existence of K\"ahler-Ricci solitons under the condition (\ref{bar4}).
As in the proof of Proposition \ref{necessary}, one can also show that  (\ref{bar4}) is a necessary condition by using the computation  as for toric manifolds  \cite{WZZ}.

\subsection{Reduction of Modified K-energy}
The following is a generalization of Proposition 3.1 in \cite{WZZ}.
\begin{prop}\label{M-energy}
Let $\phi\in\mathcal H_{K\times K}(\omega_0)$ and $u$ be the Legendre function of $\psi=\psi_0+\phi$.
Then
\begin{equation}\label{MK-energy}
\mu_{\omega_0}^X(\phi)=\frac{1}{V}\mathcal K^X(u)+const.,\nonumber
\end{equation}
where $\mathcal K^X(u)=\mathcal N^X(u)+\mathcal L^X(u)$, and
\begin{align}
\mathcal L^X(u)= \int_{2P_+}\langle y-4\rho,\nabla u\rangle e^{\theta_X(y)}\pi\,dy, \nonumber.
\end{align}
\begin{align}
\mathcal N^X(u)=-\int_{2P_+}\left(\log\det\left(u_{ij}\right)-\chi\left(\nabla u\right)-4\langle\rho,\nabla u\rangle\right)e^{\theta_X(y)}\pi\,dy. \label{MN}
\end{align}
\end{prop}

\begin{proof}
By (\ref{MA}) and (\ref{5115}), we reduce ${\frac{\omega_{\phi}^n}{\omega_0^n}}e^{\phi-h_0}$ to a function on $\mathfrak a_+$ by
$$\frac{\omega_{\phi}^n}{\omega_0^n}e^{\phi-h_0}=
\frac{MA_{\mathbb R}(\psi)\prod_{\alpha\in\Phi_+}\langle\alpha,\nabla \psi\rangle^2}{J(x)}e^\psi.$$
Then
\begin{align}
&\ \ \ {\frac 1 {C_H}}\int_M\log\left({\frac{\omega_{\phi}^n}{\omega_0^n}e^{\phi-h_0}}\right)e^{\theta_X(\omega_{\phi})}\,\omega_{\phi}^n
\nonumber\\
&=\int_{\mathfrak a_+}\left[\log MA_{\mathbb R}(\psi)+\psi+\chi(x)\right]e^{\theta_X(y)}MA_{\mathbb R}(\psi)\prod_{\alpha\in\Phi_+}\langle\alpha,\nabla \psi\rangle^2dx+C_{\pi}\nonumber\\
&=\int_{2P_+}\left[-\log\det(u_{,ij})+\chi(\nabla u)\right]e^{\theta_X(y)}\pi\,dy\nonumber\\
&\ \ \ +\int_{\mathfrak a_+}\psi e^{\theta_X(y)}MA_{\mathbb R}(\psi)\prod_{\alpha\in\Phi_+}\langle\alpha,\nabla \psi\rangle^2dx+C_{\pi},\label{311}
\end{align}
where $C_{\pi}=\int_{2P_+}\log\pi(y)\cdot e^{\theta_X(y)}\pi \,dy$
is a uniform constant. On the other hand,
\begin{align}
&\ \ \ -{\frac 1 {C_H}}\int_0^1\int_M\dot{\phi}_te^{\theta_X(\omega_{\phi_t})}\omega_{\phi_t}^n\wedge dt\notag\\
&=-\int_0^1\int_{\mathfrak a_+}\dot{\phi}_te^{\theta_X(y)}MA_{\mathbb R}(\psi_{t})\prod_{\alpha\in\Phi_+}\langle\alpha,\nabla \psi_{t}\rangle^2dx\wedge dt\notag\\
&=\int_{2P_+}ue^{\theta_X(y)}\pi\,dy-\int_{2P_+}u_{0}e^{\theta_X(y)}\pi\,dy\notag\\
&=\int_{\mathfrak a_+}\left(x^i{\frac{\partial \psi}{\partial x^i}}-\psi\right)e^{\theta_X(y)}MA_{\mathbb R}(\psi)\prod_{\alpha\in\Phi_+}\langle\alpha,\nabla \psi\rangle^2dx+C.\notag
\end{align}
Combining  this with (\ref{311}), we get
\begin{equation}
\begin{aligned}
\mu_{\omega_0}^X(\phi)&={\frac1{V}}\int_{2P_+}\langle y,\nabla u\rangle e^{\theta_X(y)}\pi\,dy\\
&\  \ +{\frac1{V}}\left(\int_{2P_+}\chi\left(\nabla u\right)e^{\theta_X(y)}\pi\,dy
-\int_{2P_+}\log\det\left(u_{,ij}\right)e^{\theta_X(y)}\pi\,dy\right)+const.\nonumber
\end{aligned}
\end{equation}
This proves the proposition.
\end{proof}

\subsection{Properness}
Analogous to Proposition \ref{5301}, we have

\begin{prop}\label{Mlinear}
Under $(\ref{bar4})$,  it holds
\begin{equation}\label{soliton-linear}
\mathcal L^X(u)\geq\lambda_X\int_{\partial(2P_+)}u\langle y,\nu\rangle e^{\theta_X(y)}\pi\,d\sigma_0,~\forall u\in\hat {\mathcal C}_{\infty,+},
\end{equation}
where $\lambda_X>0$ is a uniform constant.
\end{prop}

\begin{proof}
By  (\ref{bar4}), we have
$$\langle bar_X-4\rho,\nabla u\rangle\geq0,~\forall y\in 2P_+.$$
Then $$\mathcal L^X(u)\geq\int_{2P_+}\langle y-bar_X,\nabla u\rangle e^{\theta_X(y)}\pi\,dy.$$
On the other hand, by  the convexity of $u$, we have
$$\langle y-bar_X,\nabla u\rangle\geq u(y)-u(bar_X)\geq\langle y-bar_X,\nabla u|_{bar_X}\rangle.$$
Thus
\begin{eqnarray*}
\mathcal L^X(u)\geq\int_{2P_+}\langle y-bar_X,\nabla u|_{bar_X}\rangle e^{\theta_X(y)}\pi\,dy=0.
\end{eqnarray*}
Now we can follow the arguments in the proof of Proposition \ref{5301} to get (\ref{soliton-linear}).
\end{proof}

\begin{prop}
Under (\ref{bar4}), for any $\delta\in(0,1)$, there exists a uniform constant $C_\delta>0$ such that
\begin{equation}\label{f-x-proper}
\mathcal K^X(u)\geq\delta\int_{2P_+}ue^{\theta_X(y)}\pi\,dy-C_\delta,~\forall u\in\hat{\mathcal C}_{\infty,+}.
\end{equation}
\end{prop}
\begin{proof}

Since $-\log\det$ and $\chi(x)$ are  both   convex,   by (\ref{MN}), we have
\begin{equation}\label{4106}
\mathcal N^X(u)\geq\int_{2P_+}\left(\left.{\frac{\partial \chi}{\partial x^i}}\right|_{x=\nabla u_0}+4\rho_i\right)u_{,i} e^{\theta_X(y)}\pi\,dy
-\int_{2P_+}u_0^{ij}u_{,ij}e^{\theta_X(y)}\pi\,dy+C_0.\nonumber
\end{equation}
By integration by parts, we get an analogue of (\ref{49}),
\begin{align}
\mathcal N^X(u)\geq&
-\sum_A\int_{F_A}{\frac{2}{\lambda_{A}}}\langle y,\nu_A\rangle ue^{\theta_X(y)}\pi \,d\sigma_0\nonumber\\
&-\int_{2P_+}\left( u^{ij}_{0,ij}+u_0^{ij}c^ic^j+2u_{0,j}^{ij}c^i\right)ue^{\theta_X(y)}\pi \,dy\nonumber\\
&-\int_{2P_+}2\left(u_{0,j}^{ij}+u_0^{ij}c^j\right)\pi_{,i}e^{\theta_X(y)}udy
+\int_{2P_+}Que^{\theta_X(y)}\pi dy.\label{4107}
\end{align}
Here we used the fact that
$$\left(\left.{\frac{\partial \chi}{\partial x^i}}\right|_{x=\nabla u_0}+4\rho_i\right)\pi(y)=0,~\forall y\in \partial(2P^+)$$
and
$$c^i{\frac{\partial \chi}{\partial x^i}}(x)+4c^i\rho_i=-2\sum_{\alpha\in\Phi_+}c^i\alpha_i\cdot\coth x+4c^i\rho_i=0.
$$

On the other hand, $\mathcal L^X(u)$ can be rewritten as
$$\mathcal L^X(u)=\int_{\partial(2P_+)}\langle y-4\rho,\nu\rangle ue^{\theta_X(y)}\pi\,d\sigma_0-\int_{2P_+}[n+c^i(y_i-4\rho_i)]u e^{\theta_X(y)}\pi\,dy.$$
Note that $\theta_X(y)$ is uniformly bounded on $2P^+$.   Then we have
\begin{align}
\int_{2P_+}\langle\rho,{\nabla\pi}\rangle ue^{\theta_X(y)}dy\leq \mathcal L^X(u)+C\int_{\partial(2P_+)}u\langle y,\nu\rangle e^{\theta_X(y)}\pi\, d\sigma_0,~\forall u\in\hat{\mathcal C}_{\infty,+}.\nonumber
\end{align}
Thus by (\ref{4107}), we get
\begin{eqnarray}
\mathcal N^X(u)&\geq& C_0-C_\Lambda\int_{\partial(2P_+)}u\langle y,\nu\rangle e^{\theta_X}\pi\,d\sigma_0-C_L\mathcal L(u)+\int_{2P_+}Que^{\theta_X}\pi\,dy\nonumber\\
&:=&C_0-\mathcal L^X_B(u),~\forall u\in\hat{\mathcal C}_{\infty,+}.\label{n-l-x}
\end{eqnarray}

By Proposition \ref{Q int},  as in (\ref{605}),  we  see that  for any $0<\delta\leq1$ there is a constant $C_{\delta}>0$ independent of $u$ such that,
\begin{equation}\label{4110}
\mathcal L_B^X(u)-\mathcal L^X(u)\leq C_{\delta}\mathcal L^X(u)-\delta\int_{2P_+}ue^{\theta_X(y)}\pi\, dy.
\end{equation}
Now by  (\ref{n-l-x}) and (\ref{4110}),  (\ref{f-x-proper}) follows by the argument in the proof of Proposition.
\end{proof}

Propostion \ref{M-energy}  implies Theorem \ref{thm40}  by  the following  lemma, which can be derived in a same way as for  Lemma \ref{609} (also see \cite[Lemma 3.4]{WZZ}).

\begin{lem}\label{4112}
There exists a uniform $C_{J,X}>0$ such that
\begin{equation}
\left| J_{\omega_0,X}(\tilde\phi)-{\frac 1 {V}}\int_{2P_+}\tilde ue^{\theta_X(y)}\pi\, dy\right|\leq C_{J,X},\nonumber
\end{equation}
where $\tilde u\in\hat{\mathcal C}_{\infty,W}$ and $\psi_0+\tilde{\phi}$ is the Legendre function of $\tilde u$.
\end{lem}

\section{Minimizers of K-energy}

In this section,  we discuss the weak minimizers of  $\mathcal K(u)$
under the assumption that the reduced K-energy is proper. We will adapt the argument in \cite{ZZ08}.

\subsection{Extension of $\mathcal K(\cdot)$}
Let  $P^*$  be a  union of $P$ and its open codim-1 faces. We need to complete  the space $\hat{\mathcal C}_{\infty,W}$ of  functions   on  $2P^*$.   Consider a class of convex functions on $2P^*$  which satisfies
\begin{align}\label{c-bounded}\int_{\partial(2P_+)}u\langle y,\nu\rangle \pi\, d\sigma_0\leq \kappa~
{\rm and }~\int_{2P_+}u \langle\rho,\nabla\pi\rangle dy\leq \kappa,
\end{align}
where   $\kappa\geq 0$ is a fixed  number.  Set
\begin{align}
\mathcal{\tilde C_*}^\kappa=\{u\in C(2P^*)|~&u~\text{ is a $W$-invariant convex function on } 2P^*,~\notag\\
&\text{ which is normalized as in } (\ref{normalization-u})
 \text{ such that } (\ref{c-bounded})~{\rm holds}\},\notag
\end{align}
and $\mathcal{\tilde C_*}=\cup_{\kappa\geq 0} \mathcal{\tilde C_*}^\kappa.$
We show that each  $\mathcal{\tilde C_*}^\kappa$ is a complete space. Namely,

\begin{lem}\label{extend u}
Let $\{u_k\}\subset\mathcal{\tilde C_*}^\kappa$  be a sequence. Then there is a subsequence which converges locally uniformly to some $u\in\mathcal{\tilde C_*}^\kappa$.
\end{lem}

\begin{proof}
For any domain $\Omega\subset2P$ with ${\rm dist}(\Omega,\partial(2P))>0$, one can construct  a $2P'$ as in
the proof of Proposition \ref{Q int}  such that $\Omega\subset2P'$. By Lemma \ref{norm-convex}
 and Lemma \ref{52}, we see
 $$\int_{2P'}u  \, dy=  \#W\cdot\int_{2P'_+}u  \, dy  \leq  C_0\kappa.$$
 Thus there is a subsequence (still denoted by $\{u_k\}$) converging  locally uniformly to some $u$  on $2P'$.   Clearly $u$ is a $W$-invariant, normalized convex function on $2P'$.  Since $2P'$  exhausts $2P$,  $u\in C(2P)$. Moreover,
 $u$ satisfies (\ref{normalization-u}).
  Defining $u$ on the boundary by  $u(z):=\displaystyle\lim_{t\to1^-} u(tz),$
then $u\in\mathcal{\tilde C_*}^\kappa.$
\end{proof}

It is clear that the linear part $\mathcal L(u)$ is well-defined for $u\in \mathcal{\tilde C_*}$. To make $\mathcal N (u)$ well-defined, we  let  $\partial^2u=D^2u$ at the points where the Hessian exist,  and $\partial ^2u =0$ otherwise. This can be done since  the second derivatives of a convex function exist almost everywhere. In fact,
$\mu_r[u]=\det(\partial^2u)\,dy$ defines the regular part of the Monge-Amp\'ere measure $\mu[u]=\mu_r[u]+\mu_s[u]$ \cite{TW}, where the supporting set $\mathcal S_u$ of $\mu_s[u]$ has Lebesgue measure $0$.
We introduce
$$\mathcal N^+(u):=-\int_{2P_+}\left[\log\det(\partial^2u)+2\sum_{\alpha\in\Phi_+}\log\sinh\alpha(\partial u)-4\rho(\partial u)\right]^+\pi\, dy.$$

The following proposition guarantees that $\mathcal N(u)$ is well-defined for any  $u\in\mathcal{\tilde C_*}$.

\begin{prop}\label{finite N}
For $u\in\mathcal{\tilde C_*}$,  $\mathcal N^+(u)>-\infty$.
More precisely, for any $0<\epsilon<1$,  there is a uniform constant $C(\epsilon)$ such that
\begin{align}\label{N^+}
-\mathcal N^+(u)&\leq\epsilon\left(\int_{\partial(2P_+)}u\langle y,\nu\rangle\pi\, d\sigma_0+\int_{2P_+}(\langle\rho,\nabla \pi\rangle+\pi)u\, dy\right)+C(\epsilon).
\end{align}
\end{prop}

The following lemma can be proved as in \cite[Lemma 2.2]{ZZ08}. We  omit the proof.

\begin{lem}\label{702}
Let $u\in\mathcal{\tilde C_*}$ and $\{u_k\}\subset\mathcal{\tilde C_*}$ be a sequence of convex functions which  converges locally uniformly to $u$ with $\partial u_k\to\partial u,\,\partial^2u_k\to\partial^2u$
almost everywhere. Suppose that
 \begin{align}\label{condition-regularity}\alpha(\partial u_k),\alpha(\partial u)\geq\epsilon_0>0,\,\forall\alpha\in\Phi_{+}\text{ and }\det(\partial^2 u_k),\det(\partial^2 u)\geq\epsilon_0>0.
 \end{align}
 Then for any $\Omega\Subset 2P,$
$$\begin{aligned}
&\ \ \ \int_{\Omega}(\chi(\partial u)+4\rho(\partial u))\pi\, dy-\int_{\Omega}\log\det(\partial^2u)\pi\, dy\\&=\lim_{k\to\infty}\left[\int_{\Omega}(\chi(\partial u_k)+4\rho(\partial u_k))\pi\, dy-\int_{\Omega}\log\det(\partial^2u_k)\pi\, dy\right].\end{aligned}$$
\end{lem}

For any $u\in\mathcal{\tilde C_*}$, we can  replace it  by $\tilde u(y):=u(y)+{\frac12}c|y|^2+\rho(y)$, where $c$ is sufficiently large such that
\begin{equation}\label{+704}
\det(\partial^2\tilde u)\geq c^n,~\alpha(\partial\tilde u)>\alpha(\rho)> 0,
\end{equation}
and
\begin{equation}\label{+704'}
\log\det(\partial^2\tilde u)-\chi(\partial \tilde u)-4\rho(\partial \tilde u)>n\log c-2\sum_{\alpha\in\Phi_+}\log\left({\frac{1-e^{-2\alpha(\rho)}}2}\right)>0.
\end{equation}
Then $-\mathcal N^+(u)<-\mathcal N^+(\tilde u)$. Thus $\tilde u$   satisfies (\ref{condition-regularity})  and  we need to estimate   $\mathcal N^+(\tilde u)$.
\begin{proof}[Proof of Proposition \ref{finite N}]
We first show Proposition  \ref{finite N}  is true for $u\in \mathcal{\tilde C_*}\cap C(\overline{2 P})$.
For any $\delta>0$, let $P^{\delta}:=(1-\delta)P$ be a dilated polytope and $P^{\delta}_+:=P^{\delta}\cap\bar{\mathfrak a}_+$. Define a family of smooth functions $u_h(y)=h^{-r}\int_{2P}\vartheta(h^{-1}(y-z))u(z)dz$ for small $h>0$ and $y\in2P^\delta$. Here $\vartheta(\cdot)$ is a support function in $\mathbf B_O(1)$ such that $\int_{\mathbf B_O(1)}\vartheta=1$. It is easy to see that $u_h$ is convex and $W$-invariant. Moreover, $\partial u_h\to\partial u$ and $\partial^2u_h\to\partial^2u$ almost everywhere.

For $\tilde {u}_h=u_h+{\frac12}c|y|^2+\rho(y)$, by (\ref{+42}) and integration by parts, we have
\begin{align}\label{710}
&\ \ \int_{2P_+^{\delta}}\left(\log\det(\partial^2\tilde{u}_h)-\chi(\partial \tilde {u}_h)-4\rho(\partial\tilde {u}_h)\right)\pi\, dy\notag\\
&\leq%-\int_{\partial(2P_+^{\delta})}(\left.{\frac{\partial\chi}{\partial x^i}}\right|_{x=\nabla u_0}+4\rho_i+ u_{0}^{ij}\nu_j\pi_{,i}u_h)\nu^iu_{h}\pi d\sigma_0\notag\\
-\int_{\partial(2P_+^{\delta})}\tilde Q^i\nu_i\tilde {u}_h\pi\, d\sigma_0+\int_{\partial(2P_+^{\delta})}\left(u_{0}^{ij}\nu_j\tilde {u}_{h,i}\pi-u_{0,i}^{ij}\nu_j\tilde{u}_h\pi \right)d\sigma_0\notag\\
&\ \ \ +\int_{2P_+^{\delta}}\left[u_{0,ij}^{ij}\pi+2u_{0,j}^{ij}\pi_{,i}+4\langle\rho,\nabla \pi\rangle-Q\pi\right]\tilde {u}_hdy.
\end{align}
Here $$\tilde Q^i=\left.{\frac{\partial \chi}{\partial x^i}}\right|_{x=\nabla u_0}+4\rho_i+{\frac{u_0^{ij}\pi_{,j}}{\pi}},$$
and $Q$ is given by (\ref{59}).
Let $\xi^i=\pi u_{0}^{ij}\nu_j$. We see   that $|\xi|=O((\alpha(y))^2)$ near $W_{\alpha}$.  By the convexity of $u_h$, $$|\xi^i\tilde {u}_{h,i}(y)|\leq\max\{|\tilde {u}_h(y+\xi)-\tilde u_h(y)|,|\tilde {u}_h(y-\xi)-\tilde {u}_h(y)|\}.$$
Since $\pi=0$ on Weyl walls, %$u_h$ is uniformly continues on $2\bar P^\delta$,
we have
\begin{eqnarray*}
\int_{\partial(2P_+^{\delta})\cap W_{\alpha}}u_{0}^{ij}\nu_j\tilde {u}_{h,i}\pi\, d\sigma_0=0,~\int_{\partial(2P_+^{\delta})\cap W_{\alpha}}\left(u_{0}^{ij}\nu_j\tilde {u}_{h,i}\pi-u_{0,i}^{ij}\nu_j\tilde {u}_h\pi\right) d\sigma_0=0.
\end{eqnarray*}
By taking $h\to0$ and then $\delta\to0$ with Lemma \ref{bound-measure}, we get
$$\int_{\partial(2P_+^{\delta})}\left(u_{0}^{ij}\nu_j\tilde {u}_{h,i}\pi-u_{0,i}^{ij}\nu_j\tilde {u}_h\pi \right)d\sigma_0%={\frac1{\#W}}\int_{\partial(2P^{\delta})}\left(u_{0}^{ij}\nu_ju_{h,i}\pi-u_{0,i}^{ij}\nu_ju_h\pi \right)d\sigma_0
\to\sum_A\int_{F_A}{\frac{2}{\lambda_A}}\widetilde u\langle y,\nu_A\rangle\pi\, d\sigma_0.$$

The last term in (\ref{710}) can be settled by
(\ref{43})-(\ref{u-0-ij-2}) and Proposition \ref{Q int}.
It remains to deal with the first term involving $\tilde Q$.
In fact,  by using  the similar argument as  in the proof of Lemma  \ref{521} (checking the Cases (i)-(iii) there), we can  get $|\tilde Q^i\nu_i|\leq C_{\tilde Q}$
for some uniform $C_{\tilde Q}$ depending only on $P$ and $u_0$.
% In fact, $|\tilde Q^i\nu_i|=O(\delta)$ on $\partial (2P_+^{\delta})$ and away from Weyl walls.
Now by Lemma \ref{702}, taking $h\to0$ and then $\delta\to0$ in (\ref{710}), we get a uniform constant $C$ such that
\begin{equation}\label{711}
\begin{aligned}
-\mathcal N^+(\tilde u)&\leq (1+C'_{\tilde Q})\sum_A\int_{F_A}{\frac{2}{\lambda_A}}\tilde u\langle y,\nu_A\rangle\pi\, d\sigma_0\\
&\ \ +\int_{2P_+}\left[u_{0,ij}^{ij}\pi+2u_{0,j}^{ij}\pi_{,i}+4\langle\rho,\nabla \pi\rangle-Q\pi\right]\tilde udy
+C(u_0)\\&\leq C\left(\int_{\partial(2P_+)}\tilde u\langle y,\nu\rangle\pi\, d\sigma_0+\int_{2P_+}(\langle\rho,\nabla \pi\rangle+\pi)\tilde u\, dy\right)+C(u_0).\nonumber
\end{aligned}
\end{equation}
Replacing $u$ by $\epsilon u$, we obtain (\ref{N^+}).

 For a  general  $u\in\tilde{\mathcal C}_*$, we consider
$u^t(\cdot)=u(t\cdot)$ for $0<t<1$. Then $\partial u^t\to\partial u\text{ and }\partial^2u^t\to\partial^2u$ almost everywhere when $t\to1^-$.  Since  $u^t\in C(\overline{2P})$,  $(\ref{N^+})$ holds for all $u^t$.  Note that the constants in $(\ref{N^+})$ are independent of $t$.   Thus  the proposition is proved.
\end{proof}

\subsection{The existence of minimizers}

We prove that $\mathcal K(\cdot)$ is lower semi-continuous on $\mathcal{\tilde C_*}$. Namely,

\begin{prop}\label{semi-continuity}
Suppose that $\{u_k\}\subset\mathcal{\tilde C_*}^\kappa$ converges locally uniformly to $u\in\mathcal{\tilde C_*}^\kappa$ for some $\kappa>0$, and
%\begin{equation}\label{N lower bound un}
$\mathcal N(u_k)<C_0$
%\end{equation}
for some constant $C_0$. Then
\begin{equation}\label{N lower bound}
\mathcal N(u)<+\infty
\end{equation}
and there exists a subsequence of $\{u_k\}$ such that
\begin{equation}\label{semi-c}
 \mathcal K(u)\leq \liminf_{k\to\infty}\mathcal K(u_k).
\end{equation}
\end{prop}

We will modify the proofs in \cite[Section 3]{ZZ08}.  The proof is divided into several steps. First, we have

\begin{lem}\label{705}
Suppose that $\{u_k\}\subset\mathcal{\tilde C_*}^\kappa$ converges locally uniformly to $u\in\mathcal{\tilde C_*}^\kappa$ for some $\kappa>0$.   Then for any $\delta>0$, we have
\begin{equation}\label{706}
\limsup_{k\to\infty}\int_{2P^{\delta}}\log\det(\partial^2u_k)\pi\, dy\leq\int_{2P^{\delta}}\log\det(\partial^2u)\pi\, dy
\end{equation}
and
\begin{align}\label{707}
&\ \ \ \limsup_{k\to\infty}\int_{2P^{\delta}}\left[\log\sinh\alpha(\partial u_k)-\alpha(\partial u_k)\right]\pi\, dy\notag\\
&\leq\int_{2P^{\delta}}\left[\log\sinh\alpha(\partial u)-\alpha(\partial u)\right]\pi\, dy.
\end{align}
\end{lem}

\begin{proof}
 (\ref{706}) can be proved as the same as \cite[Lemma 3.1]{ZZ08}. Here we give an alternative proof.
  Let $\mathcal S$ be a  union of  supports sets $\mathcal S_u$ and all $\mathcal S_{u_k}$. Then $\forall\epsilon'>0$, there is a closed subset $\Omega_{\epsilon'}\subset 2P^{\delta}_+\backslash \mathcal S$ such that $\int_{2P_+^\delta\backslash\Omega_{\epsilon'}}\pi\,dy<\epsilon'$.  We observe (cf. \cite[Proposition 3.1]{BB}),
$$\begin{aligned}-\int_{\Omega_{\epsilon'}}\log\det(\partial^2u)\pi \,dy&=\sup_{f\in C(2P_+^\delta)}\left(\int_{\Omega_{\epsilon'}}f\pi \,dy-\log\int_{\Omega_{\epsilon'}}e^f\det(\partial^2u)\pi\, dy\right)\\
&\ \  \ -\log\left(\int_{\Omega_{\epsilon'}}\det(\partial^2u)\pi \,dy\right)\int_{\Omega_{\epsilon'}}\pi \,dy.\end{aligned}$$
Then for a fixed function $f$, by the upper semi-continuity of Monge-Amp\'ere measure,
$$-\log\int_{\Omega_{\epsilon'}}e^f\det(\partial^2u)\pi \,dy$$ is lower semi-continuous as a functional of $u$.
 Thus
\begin{align}\label{sequence-convergence}\limsup_{k\to\infty}\int_{\Omega_{\epsilon'}}\log\det(\partial^2u_k)\pi\, dy\leq\int_{\Omega_{\epsilon'}}\log\det(\partial^2u)\pi\, dy.
\end{align}
On the other hand, since osc$u_k$ are uniformly bounded on $2P^{\frac{\delta}{2}}$, we have
$$\int_{\Omega_{\epsilon'}}\det(\partial^2u_k)\pi \,dy\leq\int_{2P_+^\delta}\det(\partial^2u_k)\pi \,dy\leq C_0\left({\frac{\text{osc} u_k}{\delta}}\right)^r<\infty,$$
where $C_0$ is independent of $k$. Then by the concavity of $\log$,
\begin{align}\label{sequence-small}\int_{2P_+^\delta\backslash\Omega_{\epsilon'}}\log\det(\partial^2u_k)\pi\,dy
\leq\int_{2P_+^\delta\backslash\Omega_{\epsilon'}} \pi dy\cdot\log\left[ \frac{C_0(\frac{\text{osc} u_k}{\delta})^r}{  \int_{2P_+^\delta\backslash\Omega_{\epsilon'}} \pi dy}\right].
\end{align}
Combining (\ref{sequence-convergence}) and (\ref{sequence-small}), we have
$$\limsup_{k\to\infty}\int_{2P^\delta_+}\log\det(\partial^2u_k)\pi\,dy\leq\int_{\Omega_{\epsilon'}}\log\det(\partial^2u)\pi\, dy+\epsilon'\log\left[ \frac{C_0(\frac{\text{osc} u_k}{\delta})^r}{\epsilon'}\right],$$
letting $\epsilon'\to0$,  we get (\ref{706}).

(\ref{707}) follows from  Fatou's Lemma.
\end{proof}

\begin{proof} [Proof of Proposition \ref{semi-continuity}]. First  we use a contradiction argument to prove (\ref{N lower bound}).
Suppose $\mathcal N(u)=+\infty$. Then for any $C>0$ there exists a $\delta_C>0$ such that
$$
-\int_{2P_+^\delta}\left[\log\det(\partial^2u)-\chi(\partial u)-4\rho(\partial u))\right]\pi \,dy\geq C,~\forall~ 0\leq\delta<\delta_C.
$$
Thus by Lemma \ref{705}, for any $\epsilon>0$, there exists an $k_{\epsilon,\delta}$ such that
$$
-\int_{2P_+^\delta}\left[\log\det(\partial^2u_k)-\chi(\partial u_k)-4\rho(\partial u_k))\right]\pi \,dy\geq C-\epsilon,~\forall ~k\geq k_{\epsilon,\delta}.
$$
Together with the assumption of $\mathcal N(u_k)<C_0$, we get
$$
\int_{2P_+\backslash 2P_+^\delta}\left[\log\det(\partial^2u_k)-\chi(\partial u_k)-4\rho(\partial u_k))\right]\pi \,dy\geq C-C_0-\epsilon.
$$
On the other hand, by (\ref{N^+}), we  also have
$$\int_{2P_+\backslash 2P_+^\delta}\left[\log\det(\partial^2u_k)-\chi(\partial u_k)-4\rho(\partial u_k)\right]\pi \,dy\leq-\mathcal N^+(u_k)\leq C'\kappa$$
for some uniform $C'$. Hence we get a contradiction since  the constant $C$ can be taken sufficiently  large.
(\ref{N lower bound})  is true.

\vskip 6pt
Next we prove (\ref{semi-c}). Since  the linear part $\mathcal L(\cdot)$ of $\mathcal K(\cdot)$  is lower semi-continuous,  it suffices to deal with
the nonlinear part $\mathcal N(\cdot)$. Observe
\begin{align}
\mathcal N(u)-\mathcal N(u_k)&=\int_{2P_+\backslash 2P_+^\delta}\left[\log\det(\partial^2u_k)-\chi(\partial u_k)-4\rho(\partial u_k)\right]\pi\, dy\nonumber\\
&\ \ \ -\int_{2P_+\backslash 2P_+^\delta}\left[\log\det(\partial^2u)-\chi(\partial u)-4\rho(\partial u)\right]\pi\, dy\nonumber\\
&\ \ \ +\int_{2P_+^\delta}\left[\log\frac{\det(\partial^2u_k)}{\det(\partial^2u)}-\chi(\partial u_k)+\chi(\partial u)+4\rho(\partial u-\partial u_k)\right]\pi\, dy\label{718}\notag\\
&:=I_1+I_2+I_3.
\end{align}
In view of (\ref{N lower bound}), for any $\epsilon>0$, there is a $\delta_\epsilon>0$ such that for any $\delta<\delta_\epsilon$,
$I_2<\epsilon$.
By Lemma \ref{705}, there is an $k_{\epsilon,\delta}>0$ such that for any $k>k_{\epsilon,\delta}$,
$I_3<\epsilon$. It remains to estimate $I_1$.

We use a scaling trick to get a similar estimate as (\ref{N^+}).
For any $\Lambda> 1$,
\begin{align}\label{scale-difference}
&\ \ \  \log\det(\partial^2 u_k)-\chi(\partial u_k)-4\rho(\partial u_k)\notag\\
&\leq\log\det\left(\frac{\partial^2 u_k}{\Lambda}\right)-\chi\left(\frac{\partial u_k}{\Lambda}\right)-4\rho\left(\frac{\partial u_k}{\Lambda}\right)+n\log\Lambda.
\end{align}
By (\ref{+42}) and integration by parts, we have
\begin{align}\label{713}
&\ \ \ \ \int_{2P_+\backslash 2P_+^\delta}\left[\log\det\left(\frac{\partial^2 u_k}{\Lambda}\right)-\chi\left(\frac{\partial u_k}{\Lambda}\right)-4\rho\left(\frac{\partial u_k}{\Lambda}\right)\right]\pi\, dy\notag\\
&\leq {\frac1{\Lambda}}\left(\int_{\partial(2P_+)}-\int_{\partial(2P^\delta_+)}\right)\left[u_0^{ij}\nu_ju_{k,i}-u_{0,j}^{ij}\nu_iu_k+\tilde Q^i\nu_iu_k\right]\pi\, d\sigma_0\notag\\
&\ \ \ +{\frac1{\Lambda}}\int_{2P_+\backslash 2P_+^\delta}\left[u_{0,ij}^{ij}\pi+2u_{0,j}^{ij}\pi_{,i}+4\langle\rho,\nabla \pi\rangle-Q\pi\right]u_k\,dy +C(u_0)\delta.\nonumber
\end{align}
Note that $|\tilde Q^i\nu_i|$ is bounded and $\{u_k\}\subset\tilde{\mathcal C}_*^\kappa$. Then,  by \cite{ZZ08}, there are $C_1, C_2>0$, such that
\begin{align}
&\ \ \left(\int_{\partial(2P_+)}-\int_{\partial(2P^\delta_+)}\right)\left[u_0^{ij}\nu_ju_{k,i}-u_{0,j}^{ij}\nu_iu_k+\tilde Q^i\nu_iu_k\right]\pi\, d\sigma_0\notag\\
&\leq C_1\sum_A{\frac2{\lambda_A}}\int_{F_A}u_k\langle y,\nu\rangle\pi\, d\sigma_0\leq C_1\kappa.\notag
\end{align}
Moreover, by  (\ref{522}) (also see Remark \ref{remark-Q}), we have
\begin{align}
& \ \  \ \int_{2P_+\backslash 2P_+^\delta}\left[u_{0,ij}^{ij}\pi+2u_{0,j}^{ij}\pi_{,i}+4\langle\rho,\nabla \pi\rangle-Q\pi\right]u_k\,dy\notag\\
&\leq C_2\int_{2P_+\backslash 2P_+^\delta}u_k(\pi+\langle\rho,\nabla \pi\rangle) \,dy\leq C_2\kappa.\notag
\end{align}
Thus
$$\int_{2P_+\backslash 2P_+^\delta}\left[\log\det\left(\frac{\partial^2 u_k}{\Lambda}\right)-\chi\left(\frac{\partial u_k}{\Lambda}\right)-4\rho\left(\frac{\partial u_k}{\Lambda}\right)\right]\pi\, dy
 \le  {\frac{(C_1+C_2)\kappa}{\Lambda}}+C(u_0)\delta.$$
Hence by (\ref{scale-difference}), we obtain
\begin{eqnarray*}
I_1\leq {\frac{(C_1+C_2)\kappa}{\Lambda}}+C(u_0)\delta+C\delta\log\Lambda.
\end{eqnarray*}
Choosing  a sufficiently large $\Lambda$ such that ${\frac{(C_1+C_2)\kappa}{\Lambda}}<\epsilon$, and  $\delta$ small enough, we get  $I_1 \leq 2\epsilon$.
The proposition is proved.

\end{proof}

\begin{proof}[Proof of
Theorem \ref{LZZ2}]  The first part follows from Proposition \ref{semi-continuity}.
For the second part, we take a minimizing sequence $\{u_k\}$ of $\mathcal K(\cdot)$ in
$\mathcal C_{\infty,+}$.  Then by Lemma \ref{609} and the properness of $\mu(\cdot)$,  there exists  a constant $\kappa$ such that the normalized sequence  $\tilde u_k$ is a subset of  $\mathcal{\tilde C_*}^\kappa$.  Moreover,
$\mathcal N(u_k)<C_0$ for some $C_0$.
 Thus by Lemma \ref{extend u},  there is a limit $u$ of a subsequence of  $\tilde u_k$ in $\mathcal{\tilde C_*}^\kappa.$  Proposition \ref{semi-continuity} implies that   $u$ is a minimizer of  $\mathcal K(\cdot)$ in  $\mathcal{\tilde C_*}.$

\end{proof}

\section{Appendix}

\begin{lem}\label{appendix}
Let  $\lambda_i ~{i=1,...,m}$ be $m$ positive real numbers and $c_i(y)>0$ $m$ positive functions.  Let  $\alpha(y)>0$ be another positive function such that
$$ {\frac{\alpha(y)}{c_{i}(y)}}\le \epsilon_0<<1,~i=1,...,m.$$
Then
\begin{equation}\label{matrix-l}
\Delta:=\left(\sum_{i}\frac{1}{c_i(y)
+\lambda_i\alpha(y)}\right)^{-1}
-\left(\sum_i{\frac{1}{c_{i}(y)}}\right)^{-1}
=O(\alpha(y)).
\end{equation}
\end{lem}

\begin{proof}
Denote $I=\{1,...,m\}$. Since
\begin{eqnarray*}
\left(\sum_{i\in I}{\frac{1}{{c_{i}
+\lambda_i\alpha}}}\right)^{-1}={\frac{\prod_{k\in I}(c_k+\lambda_k\alpha)}{\sum_{i\in I}\prod_{j\not=i}(c_j+\lambda_j\alpha)}},~
\end{eqnarray*}
\begin{eqnarray*}
\Delta={\frac{\prod_{k\in I}(c_k+\lambda_k\alpha)\sum_{i\in I}\prod_{j\not=i}c_i-\prod_{k\in I}c_k\sum_{i\in I}\prod_{j\not=i}(c_j+\lambda_j\alpha)}{\sum_{i,k\in I}\prod_{j\not=i}\prod_{l\not=k}c_j(c_{l}
+\lambda_l\alpha)}}=:{\frac{\Delta_1}{\Delta_2}}.
\end{eqnarray*}
By a direct computation,  we have
$$\Delta_1=\sum_{i\in I}\prod_{j\not=i}c^2_j
\lambda_i\alpha
+\sum_{l=2}^{m}\sum_{i_1,...,i_{2m-1-l}}\lambda'_{i_1...i_{2m-1-l}}c_{i_1}...c_{i_{2m-1-l}}\alpha^l,$$
where $\lambda'_*$ are constants and $c_{i_1}...c_{i_{2m-1-l}} $ is a $(2m-1-l)$ product of $c_k$  of the form
$$\prod_{i'\in \{i_1,..., i_{m-l}\}\subset I}c_{i'}\prod_{j'\in \{i_1,...,i_{m-1}\}\subset I}c_{j'}\ \  \ {\rm or}~ \prod_{i=1,..., m}c_{i}
\prod_{j'\in\{i_1,...,i_{m-l-1}\}\subset I}c_{j'}.$$
Similarly,
$$\Delta_2=\sum_{i,k\in I}\prod_{j\not=i,l\not=k}c_jc_l
+\sum_{l=1}^{m-1}\sum_{i_1,...,i_{2m-2-l}}\lambda''_{i_1...i_{2m-2-l}}c_{i_1}...c_{i_{2m-2-l}}\alpha^l,
$$
where $\lambda''_*$ are constants and $c_{i_1}...c_{i_{2m-2-l}}$ is a $(2m-2-l)$ product of $c_k$  of  the form
$$\prod_{i'\in\{i_1,..., i_{m-l-1}\}\subset I}c_{i'}\prod_{j'\in\{i_1,...,i_{m-1}\}\subset I}c_{j'}.$$
Then  one can show
$$0<\Delta\le {\frac{\sum_{i\in I}\prod_{j\not=i}c^2_j(y)
\lambda_i\alpha(y)(1+o(1))}{\sum_{i,k\in I}\prod_{j\not=i,l\not=k}c_j(y)c_l(y)}}=O(\alpha(y)).
$$
The lemma is proved.
\end{proof}

\end{document}